\documentclass[11pt,leqno]{amsart}
\usepackage{etex}

\usepackage[utf8]{inputenc} 	      
\usepackage[french,english]{babel}  
\usepackage{ucs}             	      

\usepackage{graphicx}         
\usepackage{xcolor}           
\usepackage{enumitem}         
\usepackage{tikz}
\usetikzlibrary{fit,calc,positioning,decorations.pathreplacing,matrix, shapes}
\usetikzlibrary{trees}

\usepackage{amsmath, amsthm}  
\usepackage{amsfonts, amssymb}
\usepackage{mathtools}
\usepackage{stmaryrd}         
\usepackage{bm}               

\usepackage[ breaklinks,  
            pagebackref,           
            pdfborder={0 0 0}
            ]{hyperref}	 
\hypersetup{ pdfauthor={Bérénice Delcroix-Oger} }

\usepackage[nohints]{minitoc} 

\theoremstyle{definition}
\newtheorem{definition}{Definition}[subsection]
\newtheorem{remark}[definition]{Remark}
\newtheorem{example}[definition]{Example}
\newtheorem{examples}[definition]{Examples}

\newtheorem{notation}[definition]{Notation}
\theoremstyle{plain}
\newtheorem{lemma}[definition]{Lemma}
\newtheorem{proposition}[definition]{Proposition}
\newtheorem{corollary}[definition]{Corollary}
\newtheorem{theorem}[definition]{Theorem}
	
\DeclareMathOperator{\prim}{Prim}

\def\Prim{\mathrm{Prim\, }}
\def\Hom{\mathrm{Hom}}
\def\Im{\mathop{\rm Im }}

\usepackage[color=blue!50]{todonotes}

\newcommand{\Hc}{ \mathcal{H} }
\newcommand{\A}{ \mathcal{A} }
\newcommand{\C}{ \mathcal{C} }
\newcommand{\F}{ \mathcal{F} }
\newcommand{\K}{ \mathbb{K} }
\newcommand{\comm}{ \operatorname{Comm} }
\newcommand{\id}{ \operatorname{id} }

\usepackage{xargs}

\newcommandx*\As[1][1=1]{\begin{tikzpicture}[scale=0.5]
\draw[fill=white](0,0) circle(0.5);
\draw(0,0) node{$#1$};
\end{tikzpicture}}

\newcommandx*\Av[3][1=1,2=2,3=3]{\begin{tikzpicture}[scale=0.5]
\draw(1,1)--(0,0)--(-1,1);
\draw[fill=white](0,0) circle(0.5);
\draw[fill=white](1,1) circle(0.5);
\draw[fill=white](-1,1) circle(0.5);
\draw(0,0) node{$#1$};
\draw(1,1) node{$#3$};
\draw(-1,1) node{$#2$};
\end{tikzpicture}}

\newcommandx*\Al[3][1=1,2=2,3=3]{\begin{tikzpicture}[scale=0.5]
\draw(0,0)--(0,2);
\draw[fill=white](0,0) circle(0.5);
\draw[fill=white](0,1.2) circle(0.5);
\draw[fill=white](0,2.4) circle(0.5);
\draw(0,0) node{$#1$};
\draw(0,1.2) node{$#2$};
\draw(0,2.4) node{$#3$};
\end{tikzpicture}}

\usepackage{endnotes}
\usepackage{stmaryrd}
\usepackage{textcomp}

\tikzset{
    ncbar angle/.initial=90,
    ncbar/.style={
        to path=(\tikztostart)
        -- ($(\tikztostart)!#1!\pgfkeysvalueof{/tikz/ncbar angle}:(\tikztotarget)$)
        -- ($(\tikztotarget)!($(\tikztostart)!#1!\pgfkeysvalueof{/tikz/ncbar angle}:(\tikztotarget)$)!\pgfkeysvalueof{/tikz/ncbar angle}:(\tikztostart)$)
        -- (\tikztotarget)
    },
    ncbar/.default=0.5cm,
}
\tikzset{square left brace/.style={ncbar=0.5cm}}
\tikzset{square right brace/.style={ncbar=-0.5cm}}

\tikzset{round left paren/.style={ncbar=0.5cm,out=120,in=-120}}
\tikzset{round right paren/.style={\newcommand{\C}{ \mathcal{C} }ncbar=0.5cm,out=60,in=-60}}

\title[Hopf-Borel type theorem for operads]{Confluence laws and Hopf-Borel type theorem for operads}


\author{E. Burgunder \and B. Delcroix-Oger}

\thanks{The authors thank J. Mill\`es for our fruitful conversations and M. Livernet for her useful comments about our work which really helped us to improve our presentation. \\The work of the first author was supported by ANR CATHRE. The work of the second author was supported by LabEx CIMI}

\address{EB: Universit\'e Paul Sabatier\\
Institut de Math\'ematiques de Toulouse\\
118 route de Narbonne\\
F-31062 Toulouse Cedex 9 France}
\email{burgunder@math.univ-toulouse.fr}

\address{BDO: Universit\'e Paris 7 Denis Diderot \\
Institut de Recherche en Informatique Fondamentale \\
Case 7014 \\
F-75205 Paris Cedex 13 France}
\email{bdelcroix@irif.fr}
 \date{\today}

\begin{document}

\maketitle

\begin{abstract}
In 2008, Loday shed light on the existence of Hopf-Borel theorems for operads. Using the vocabulary of category theory, Livernet, Mesablishvili and Wisbauer extended such theorems to monads. In both cases, the reasoning was to start from a mixed distributive law and then to prove that it induces an isomorphism of $\mathfrak{S}$-modules to finally get a rigidity theorem. Our reasoning goes here backward: we prove that from an isomorphism of$\mathfrak{S}$-modules one can get what we called a confluence law, which generalises mixed distributive laws, and that it is enough to obtain a rigidity theorem. This enables us to show that for any operads $\mathcal{P}$ and $\mathcal{Q}$ having the same underlying $\mathfrak{S}$-module, there exists a confluence law $\alpha$ such that any conilpotent $\mathcal{P}$ co$\mathcal{Q}$-bialgebra satisfying $\alpha$ is free and cofree over its primitive elements. Our reasoning permits us to generate many new examples.
\end{abstract}

\section*{Introduction}

Distributive laws first appeared in 1969 in Beck's article \cite{Beck}. It takes its
 name from the distributivity of addition over multiplication studied in primary school. Distributive laws give a confluent way to rewrite expressions 
 mixing different products. An example of distributive laws is the one mixing the 
 commutative product and the Lie bracket in a Poisson algebra. The notion of distributive laws has been initially 
 studied by Burroni in \cite{Burroni} for algebras and by Markl in \cite{Markl2} 
 generalising this definition to operads and linking it to the topological notion of 
 Koszulness. Fox and Markl have adapted in \cite{Markl-Fox} the notion of distributive 
 laws to expressions mixing operations and cooperations: they called the obtained 
 rewriting rules \emph{mixed distributive laws}.

Let us call a $\C^c-_{\lambda}\A$-bialgebra a vector space endowed with an algebra structure encoded by an operad $\A$, a coalgebra structure encoded by an operad $\C$, or equivalently its dual  cooperad $\C^*$, with both linked through a mixed distributive law $\lambda$. Such a bialgebra satisfies, under some assumptions, a Hopf-Borel type theorem: any conilpotent $\C^c-_{\lambda}\A$-bialgebra is free and cofree over its primitive elements. Such a theorem is called a \emph{rigidity theorem}. Particular cases of rigidity theorems were proven for instance in \cite{Borel} for commutative cocommutative Hopf bialgebras, in \cite{nui} for associative coassociative bialgebras, in \cite{AsZinb} for Zinbiel coassociative bialgebras, in \cite{Foissy} for dendriform codendriform algebras and in \cite{NAPPL} for PreLie coNAP bialgebras. The general framework for this theorem was introduced by Loday in \cite{GBO}. Rigidity theorems were then further studied, for instance in \cite{Muriel} and applications can be found for example in \cite{EPLM} to compute explicit bases of algebras. While studying the general framework and its rewriting for particular symmetric operads, it became clear to the authors that the three hypotheses of this theorem had to be clarified and some further clarifications were needed in the proof.

In this article we prove that one only of the hypotheses formulated by Loday in \cite{GBO} is needed to get a rigidity theorem. This viewpoint enables us to provide an answer to a conjecture of Loday: for a given operad $\mathcal{P}$ which encoded the structure of algebra and of coalgebra of a  $\mathcal{P}^c-\mathcal{P}$ bialgebras there exists a generalisation of mixed distributive laws, called \emph{confluence laws} between the structures such that a rigidity theorem holds. In fact, by definition, to any operad is associated a family of $\mathfrak{S}_n$-modules, i.e. the underlying vector space in graduation $n$ endowed with an action of the symmetric group $\mathfrak{S}_n$. We prove here more precisely a more general result: if two operads $\mathcal{P}$ and $\mathcal{Q}$ share the same family of $\mathfrak{S}_n$-modules, then there exists a confluence law for $\mathcal{Q}^c-\mathcal{P}$ bialgebras. This enables to develop a list of brand new examples of rigidity theorem, presented in the second part of the article. We also develop the case of explicit isomorphism of $\mathfrak{S}_n$-modules between $\mathcal{P}$ and $\mathcal{P}^*$. In such a framework, we are also able to make explicit an inductive algorithm to compute the projection into the subspace of primitive elements and the confluence law associated to a bialgebra. 

This article is organised as follows. The first part focuses on the general theory and then we restrict our study to the framework where bases for operads are known. The second part develops explicit computation of rigidity theorem. 

More precisely, in the first section, we recall Loday's theorem with its three hypotheses (H0), (H1), (H2iso): for any $\mathcal{C}^c-\mathcal{A}$ bialgebra satisfying mixed distributive laws (H0), one can compute a morphism $\varphi_V: \A(V) \rightarrow \C^c(V)$, functorial in $V$ (H1). If this morphism is bijective (H2iso) then a rigidity theorem holds. We point out here that the core property is that the morphism $\varphi$ induces a family of $\mathfrak{S}_n$-modules isomorphisms $\varphi_n :\mathcal{A}(n)\to\mathcal C^*(n) $. This property induces the existence of a confluence law which is the only property needed to obtain a rigidity theorem.

Then in the second part, some explicit cases are explored. More precisely, we will fix a basis for the vector space $\mathcal{A}(n)=\C(n)$ and dualize it. We explain this construction and give a simple condition for it to satisfy the hypotheses of rigidity theorem: its compatibility to the $\mathfrak{S}$-module structure. We then list the examples in literature which can be obtained in that way before illustrating the strength of this case with a bunch of brand new examples: PreLie coPreLie, Perm coPerm, NAP coNAP, PAN coPAN, PAN coPerm, Leibniz coAssociative, Poisson coAssociative, Leibniz coZinbiel, 2as co2as or even Dipt coDipt bialgebras.

\section*{Notations}
\begin{itemize}
\item $\mathfrak{S}_n$ denotes the symmetric group of permutations on $n$ elements.
\item $\A$ and $\C$ denotes operads, $\C^*$ denotes the dual cooperad of operad $\C$.
\item An algebra structure encoded by an operad $\A$ is denoted as a $\A$-algebra.
\item A coalgebra structure encoded by an operad $\C$ (or equivalently its dual cooperad $\C^*$) is denoted as a $\C^c$-coalgebra.
\end{itemize}

\section{General Case}
We prove in this section that only one of the hypotheses for Loday's rigidity theorem on generalised bialgebras is needed to get such a theorem. Indeed, the hypothesis (H2iso) implies the hypothesis (H1) and a weaken version of (H0), which are enough to get a rigidity theorem. This new improvement in the needed hypotheses enables us to prove a conjecture of Loday and to reach many new cases.
 
 We first recall Loday's formulation of rigidity theorem before studying the intertwining between hypotheses of the theorem. We then study the case when bases are orthogonal, so that we can compute explicitly confluence laws and idempotent, which was one problem raised in \cite{GBO}. We finally give a new formulation of rigidity theorem.

 \subsection{Loday's rigidity theorem.}

Let  $\K$ be a field of characteristic $0$. Before introducing Loday's rigidity theorem, we recall from \cite{Operads} some needed definitions.

\begin{definition}[\cite{Operads}, 5.2.1] An \emph{(symmetric algebraic) operad} $\A=(\A, \gamma, \nu)$ is a $\mathfrak{S}$-module $\A={\A(n)}_{n \geq 0}$ endowed with morphisms of $\mathfrak{S}$-modules $\gamma : \A \circ \A \rightarrow \A$, called \emph{composition map} and $\eta : I \rightarrow \A$, called the \emph{unit map}, such that $\gamma$ and $\eta$ satisfy associativity and unitality:
\begin{align}
\gamma (\gamma \circ \operatorname{Id}) & = \gamma (\operatorname{Id} \circ \gamma) \\ 
\gamma (\eta \circ \operatorname{Id}) & = \gamma (\operatorname{Id} \circ \eta)
\end{align}

\end{definition}

Let $\mathcal A, \mathcal C$ be two algebraic operads. In this article, we will only consider \emph{connected} operads, i.e. such that $\A(0)=\C(0)=\emptyset$ and $\A(1)=\C(1)=\K.\operatorname{id}$ and such that $\A(n)$ and $\C(n)$ are  finite dimensional.

Let us first recall the definition of an algebra, a cooperad and a coalgebra over an operad.
\begin{definition}[\cite{Operads}, 5.2.3] \label{defcog}An \emph{algebra} over an operad $\A$ is a vector space $A$ equipped with a $\mathfrak{S}_n$-equivariant morphism $m^n_A: \A(n) \otimes A^{\otimes n} \rightarrow A $. We denote the \emph{free algebra} over an operad $\A$ whose vector space of generators is $V$ by 
\begin{equation}
\A(V) = \bigoplus_{n \geq 1} \A(n) \otimes_{\mathfrak{S}_n} \left( \underbrace{V \otimes \ldots \otimes V}_\text{$n$} \right)= \bigoplus_{n \geq 1} \A(n) \otimes_{\mathfrak{S}_n} V^{\otimes n}.
\end{equation}
\end{definition}

\begin{definition}[\cite{Operads}, 5.7.1] 
A \emph{cooperad} $\C^*$ is the data of a family of $\mathfrak{S}_n$-modules $\C^*(n)$ for every integer $n$ and of two morphisms of $\mathfrak{S}$-modules $\Delta: \C^* \rightarrow \C^* \otimes \C^*$ and $\epsilon: \C^* \rightarrow I$ (counit) satisfying some coassociativity and counitality axioms.
When $\C$ is an operad, $\C^*(n)=\Hom(\C(n), \K)$ is a cooperad, called the \emph{dual cooperad} of $\C$.
\end{definition}

\begin{definition}[\cite{Operads}, 5.7.3] \label{defcog}A \emph{coalgebra} over an operad $\C$ is a vector space $C$ equipped with a $\mathfrak{S}_n$-equivariant morphism $\gamma^n_C: \C(n) \otimes C \rightarrow C^{\otimes n} $. This definition is equivalent to the definition of a coalgebra over the cooperad $\C^*$, which is the data of a map $\Delta^n_C: C \rightarrow \C^*(n) \otimes_{\mathfrak{S}_n} C^{\otimes n}$.
The vector space of \emph{primitive elements} of the coalgebra $C$ is:
\begin{equation}
\operatorname{Prim}(C)=\F_1C:=\left\lbrace x \in C | \delta(x)=0 \text{ for any } \delta \in \C(n),n>1 \right\rbrace.
\end{equation}
We denote the \emph{free (conilpotent) coalgebra} over an operad $\C$ (or equivalently its associated dual cooperad $\C^*$) whose vector space of primitives is $V$ by 
\begin{equation}
\C^c(V)=\bigoplus_{n \geq 1} \C^*(n) \otimes_{\mathfrak{S}_n} V^{\otimes n}.
\end{equation}
\end{definition}

\begin{notation}\label{notation}
We put the emphasis on the notation of the maps which will be used later on: an algebra over the operad $\A$ is a vector space $H$ endowed with a map $m_H : \A(H) \rightarrow H$ and a coalgebra over the cooperad $\C^*$ is a vector space $H$ endowed with a map $\delta_H : H \rightarrow \C^*(H)$. 
\end{notation}

Following \cite{Markl-Fox}, we define the notion of mixed distributive laws: 

\begin{definition}  \label{ml} 
A \emph{compatibility relation}  $\lambda$ is defined as a sequence $\{\lambda(m,n)\}$ of maps
\begin{multline*}
\lambda(m,n): \C(m)\otimes \A(n) \to \\ \oplus \A(t_1)\otimes\cdots\otimes \A(t_m)\otimes_{ \mathfrak{S}_{t_1}\times\cdots\times \mathfrak{S}_{t_m}} \mathbb K[\mathfrak{S}_N]\otimes_{\mathfrak{S}_{s_1}\times \cdots \times \mathfrak{S}_{s_n}} \C(s_1)\otimes \cdots\otimes \mathcal{C}(s_n),
\end{multline*}
where the summation is taken over all the $N\geq 1$ and $s_1+\cdot+s_n=t_1+\cdots+t_m=N$.
A compatibility relation is a \emph{mixed distributive law} if 
\begin{itemize}
\item it is compatible with the action of the symmetric group $\mathfrak{S}_m \times \mathfrak{S}_n$: 
\begin{equation*}
\lambda(m,n)(c^{\widetilde{\sigma}}_m \otimes a_n^\sigma) = \sum   a_{j_{\tilde{\sigma}(1)}} \otimes \cdots \otimes a_{j_{\tilde{\sigma}(m)}} \otimes \mu \otimes c_{i_{\sigma(1)}} \otimes \ldots \otimes c_{i_{\sigma(n)}},
\end{equation*}
where $\sigma \in \mathfrak{S}_n$ and $\widetilde{\sigma} \in \mathfrak{S}_m$ act according to the action defined on the operad as a symmetric operad and $\lambda(m,n)(c_m \otimes a_n)$ is denoted by $\sum  a_{i_1} \otimes \ldots \otimes a_{i_m} \otimes \mu \otimes c_{j_1} \otimes \cdots \otimes c_{j_n}$,
\item it is compatible with the operad structures of $\A$ and $\C$.
\end{itemize}
 The algebra $\mathcal{A}(\mathcal{C}^c(V))$ is then endowed with a structure of $\mathcal{C}^c$-coalgebra and the coalgebra $\mathcal{C}^c(\mathcal{A}(V))$ is then endowed with a structure of $\mathcal{A}$-algebra.
\end{definition}

\begin{example}The classical law for commutative co-commutative bialgebras, called \emph{non-unitary Hopf relation} is given by the mixed distributive law:
\begin{align*}
\lambda(2,2): e_2 \otimes e_2 \mapsto &\left(e_1 \otimes e_1\right) \otimes_{\mathfrak{S}_1 \times \mathfrak{S}_1} \id \otimes_{\mathfrak{S}_1 \times \mathfrak{S}_1} \left(e_1 \otimes e_1\right) \\
&+ \left(e_1 \otimes e_1\right) \otimes_{\mathfrak{S}_1 \times \mathfrak{S}_1} (12) \otimes_{\mathfrak{S}_1 \times \mathfrak{S}_1} \left(e_1 \otimes e_1\right) \\
&+\left(e_1 \otimes e_2\right) \otimes_{\mathfrak{S}_1 \times \mathfrak{S}_2} \id \otimes_{\mathfrak{S}_2 \times \mathfrak{S}_1} \left(e_2 \otimes e_1\right) \\
&+\left(e_2 \otimes e_1\right) \otimes_{\mathfrak{S}_2 \times \mathfrak{S}_1} (23) \otimes_{\mathfrak{S}_2 \times \mathfrak{S}_1} \left(e_2 \otimes e_1\right)\\
&+\left(e_2 \otimes e_1\right) \otimes_{\mathfrak{S}_2 \times \mathfrak{S}_1} \id \otimes_{\mathfrak{S}_1 \times \mathfrak{S}_2} \left(e_1 \otimes e_2\right) \\
&+\left(e_1 \otimes e_2\right) \otimes_{\mathfrak{S}_1 \times \mathfrak{S}_2} (12) \otimes_{\mathfrak{S}_1 \times \mathfrak{S}_2} \left(e_1 \otimes e_2\right) \\
&+\left(e_2 \otimes e_2\right) \otimes_{\mathfrak{S}_2 \times \mathfrak{S}_2} (23) \otimes_{\mathfrak{S}_2 \times \mathfrak{S}_2} \left(e_2 \otimes e_2\right) ,
\end{align*} 
denoting by $(e_n)$ the usual basis of $\operatorname{Comm}(n)=\A(n)=\C(n)$.
As there are no ambiguity on operations, mixed distributive law can be represented through a clearer diagram with 
\resizebox{0.7cm}{!}{$
\overbrace{\begin{tikzpicture}[scale=0.4,thick]
\draw (0,0)--(0,1);
\draw (-1,2)--(0,1)--(1,2);
\draw (0,1.75) node{\ldots};
\end{tikzpicture}}^{n}$} the usual basis of $\A(n)$ and  \resizebox{0.7cm}{!}{$
\underbrace{\begin{tikzpicture}[scale=0.4,thick]
\draw (0,0)--(0,-1);
\draw (-1,-2)--(0,-1)--(1,-2);
\draw (0,-1.8) node{\ldots};
\end{tikzpicture}}_{n}$} the usual basis of $\C(n)$:
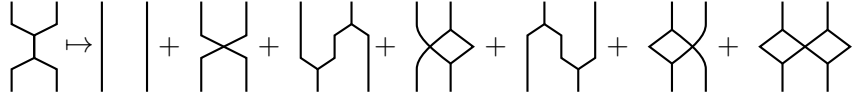
\begin{center}
\begin{tikzpicture}[scale=0.3,thick]
\draw (0,0)--(0,1)--(1,1.5)--(1,2.5)--(0,3)--(0,4);
\draw (2,0)--(2,1)--(1,1.5)--(1,2.5)--(2,3)--(2,4);
\draw (3,2) node{$\mapsto$};
\draw (4,0)--(4,4);
\draw (6,0)--(6,4);
\draw (7,2) node{$+$};
\end{tikzpicture}
\begin{tikzpicture}[scale=0.3,thick]
\draw (4,0)--(4,1.5)--(6,2.5)--(6,4);
\draw (4,4)--(4,2.5)--(6,1.5)--(6,0);
\draw (7,2) node{$+$};
\end{tikzpicture}
\begin{tikzpicture}[scale=0.3,thick]
\draw (8,4)--(8,1.5)--(8.75,1)--(8.75,0);
\draw (8.75,1)--(9.5,1.5)--(9.5,2.5)--(10.25,3)--(10.25,4);
\draw (10.25,3)--(11,2.5)--(11,0);
\draw (11.75,2) node{$+$};
\end{tikzpicture}
\begin{tikzpicture}[scale=0.3,thick]
\draw (21.5,0)--(21.5,1.25)--(22.5,2)--(21.5,2.75)--(21.5,4);
\draw (21.5,1.25)..controls (20.5,2) and (20,2.5) .. (20,3)--(20,4);
\draw (21.5,2.75)..controls (20.5,2) and (20,1.5).. (20,1)--(20,0);
\draw (23.5,2) node{$+$};
\end{tikzpicture}
\begin{tikzpicture}[scale=0.3,thick]
\draw (3,4)--(3,1.5)--(2.25,1)--(2.25,0);
\draw (2.25,1)--(1.5,1.5)--(1.5,2.5)--(0.75,3)--(0.75,4);
\draw (0.75,3)--(0,2.5)--(0,0);
\draw (4,2) node{$+$};
\end{tikzpicture}
\begin{tikzpicture}[scale=0.3,thick]
\draw (17,0)--(17,1.25)--(16,2)--(17,2.75)--(17,4);
\draw (17,1.25)..controls (18,2) and (18.5,2.5) .. (18.5,3)--(18.5,4);
\draw (17,2.75)..controls (18,2) and (18.5,1.5).. (18.5,1)--(18.5,0);
\draw (19.5,2) node{$+$};
\end{tikzpicture}
\begin{tikzpicture}[scale=0.3,thick]
\draw (1,0)--(1,1.25)--(0,2)--(1,2.75)--(1,4);
\draw (1,2.75)--(3,1.25)--(3,0);
\draw (1,1.25)--(3,2.75)--(3,4);
\draw (3,1.25)--(4,2)--(3,2.75);
\end{tikzpicture}
\end{center}

As a mixed distributive law is compatible with the operad structure of $\A$ and $\C$ is is enough to give it on generators of these operads. The other laws are then deduced from it. For instance, $\lambda(2,3)$ maps $e_3 \otimes e_2$ to a sum of $25$ terms ($6$ with  $\{t_1, t_2\} = \{2,1\}$ and $\{s_1, s_2, s_3\} = \{1,1,1\}$, $6$ with  $\{t_1, t_2\} = \{3,1\}$ and $\{s_1, s_2, s_3\} = \{2,1,1\}$, $6$ with  $\{t_1, t_2\} = \{2,2\}$ and $\{s_1, s_2, s_3\} = \{2,1,1\}$, $6$ with  $\{t_1, t_2\} = \{3,2\}$ and $\{s_1, s_2, s_3\} = \{2,2,1\}$ and $1$ with  $\{t_1, t_2\} = \{3,3\}$ and $\{s_1, s_2, s_3\} = \{2,2,2\}$).
\end{example}

We call $\C^c-_{\lambda}\A$-bialgebra any bialgebra which is a $\A$ -algebra, a $\C^c$-coalgebra and such that products and coproducts satisfy the mixed distributive laws $\lambda$. We need moreover the following notion:

\begin{definition}
The cofiltration $\F_n\Hc$ can be defined on any $\C^c$-coalgebra $\Hc$: 
\begin{equation*}
\F_n\Hc = \{x \in \Hc | \forall p >n, \forall \delta \in \C(p), \delta(x)=0\}.
\end{equation*}. The vector space $\F_1\Hc$ is the vector space of \emph{primitive elements}. Moreover, we denote by $\iota_\Hc : \mathcal{F}_1\Hc \rightarrow \Hc$ the canonical inclusion.

A $\C^c$-coalgebra $\Hc$ is said to be \emph{conilpotent} if $\Hc = \cup_{n \geq 1} \F_n\Hc$.
\end{definition}

\begin{remark}
When there will be no ambiguity, we will only write $\F_n$.
\end{remark}

The rigidity theorem as stated by Loday is:

\begin{theorem}[\cite{GBO}]
Let $\C^c-_{\lambda}\A$ be a bialgebra type which satisfies
\begin{description}
\item[(H0)] the compatibility relations $\lambda$ are distributive,
\item[(H1)] the free $\mathcal{A}$-algebra is naturally a $\C^c-_{\lambda}\A$-bialgebra,
\item[(H2iso)] the $\C^c$-coalgebra map $\varphi(V):\A(V) \rightarrow \C^c(V)$ is an isomorphism.
\end{description}
Then any conilpotent $\C^c-_{\lambda}\A$-bialgebra $\mathcal{H}$ is free and cofree over its primitive elements
\begin{equation*}
\A(\operatorname{Prim} \mathcal{H})\cong \mathcal{H} \cong \C^c(\operatorname{Prim} \mathcal{H}).
\end{equation*}
\end{theorem}

\begin{remark}
Thus the space of irreducible elements (i.e. elements which cannot be written as a linear combination of some products of elements) and the space of primitive elements are the same.
\end{remark}

\subsection{Confluence laws.}

The reasoning in Loday's theorem is to first consider a mixed distributive law thanks to which one can compute a map sending $\A(V)$ to $\C^c(V)$. To use this theorem, one then has to prove that the induced morphism is an isomorphism.

We adopt here another reasoning: we start from the data of such an isomorphism and prove that we get back the hypotheses of the theorem, more precisely, we prove the existence of an associated confluence law. Our reasoning is splitted on the next five subsections: in this subsection, we introduce the notion of confluence laws. We then explain why the hypothesis on $\varphi$ cannot be reduced in the subsection \ref{inj}. From an isomorphism $\varphi$ we get a homogeneous confluence law in the subsection \ref{hom}, before reaching confluence laws in the subsection \ref{Equiv} and stating the equivalence between conditions of the rigidity theorem. We finally state our main result in the subsection \ref{rig}.

We introduce in this subsection a generalisation of mixed distributive laws, that we call \emph{confluence laws}:

\begin{definition} \label{gml1}
A \emph{confluence law} $\alpha$ between operads $\A$ and $\C$ is a family of maps 
\begin{equation}
\alpha_{m,n} :\mathcal C(m)\otimes \mathcal A(n) \to  \mathcal A^{\otimes m}(n),
\end{equation}
such that $\mathcal A^{\otimes m}(n)$ is a short cut for $\bigoplus_{n_1+\ldots+n_m=n} \A(n_1) \otimes \ldots \otimes \A(n_m)$ and $\alpha_{m,n}$ is compatible with the structure of operad of $\C$ and with the action of the symmetric group $\mathfrak{S}_m \times \mathfrak{S}_n$, with $\mathfrak{S}_m$ acting on $\C(m)$ and $\mathfrak{S}_n$ on $\A(n)$.
\end{definition}

It is clear that mixed distributive laws as defined in Definition \ref{ml} are confluence laws, by keeping only terms involving the trivial operation of $\C(1)$ and no other operations of $\C$. On the other hand, a confluence law can be obtained by considering in mixed distributive laws coefficients which could depend on the cofiltration, which is not allowed in usual mixed distributive laws.

\begin{example}
An example of a new confluence law is presented in Equation \eqref{RelCompPL}. 
\end{example}

\begin{remark}
Confluence laws are equivalent to the data of a $\mathfrak{S}$-module morphism $\alpha : \A \rightarrow \C^*\A$ as presented in \cite{Muriel}. The equivalence comes from the equality:
\begin{equation*}
\Hom_{\mathfrak{S}_n}\left( \A(n), \C^*(m) \otimes_{\mathfrak{S}_m} \A^{\otimes m}(n)\right)=\Hom_{\mathfrak{S}_m \times \mathfrak{S}_n}\left( \C(m) \otimes\A(n), \A^{\otimes m}(n)\right),
\end{equation*}
coming from the definition of duality.

We thank again M. Livernet for helping us clearing this link and clarifying the definition of $\C^c-_{\alpha}\A$-bialgebras.
\end{remark}

Let us now define $\C^c-_{\alpha}\A$-bialgebras.

\begin{definition}\label{defBialg}
Refering to notations of Notation \ref{notation}, a $\C^c-_{\alpha}\A$-bialgebra $H$ is a $\mathbb{K}$-vector space $H$ endowed with a structure of $\A$ -algebra, a structure of $\C^c$-coalgebra, such that the following diagram commutes:

\begin{tikzpicture}
  \matrix (m) [matrix of math nodes,row sep=3em,column sep=4em,minimum width=2em] {
     \A(\mathcal{F}_1 H) & \A(H) & H & \C^*(H)  \\
 \A(H) & & & \C^*\A(H)\\};
  \path[-stealth]
    (m-2-1) edge node [below] {$\alpha$} (m-2-4)
        (m-1-1)    edge node [above] {$\mathcal{A}(\iota_H)$} (m-1-2)
    (m-1-1) edge node [left] {$\mathcal{A}(\iota_H)$} (m-2-1)
    (m-1-2) edge node [above] {$m_H$} (m-1-3)
    (m-1-3) edge node [above] {$\delta_H$} (m-1-4)
    (m-2-4) edge node [right] {$\C^*(m_H)$} (m-1-4);
\end{tikzpicture}
\end{definition}

Now we have defined the objets on which we will be working, let us look at them more closely.

\subsection{Injectivity on $\varphi$.}\label{inj}

In \cite{GBO}, a condition on $\varphi$ for the existence of a rigidity theorem is that the morphism $\varphi_V$ is an isomorphism for every vector space $V$. We first give an example of the necessity of the injectivity of $\varphi$ before studying the link between filtration and cofiltration in the associated bialgebras. This link is crucial because, as we will see later, the proof of the main theorem relies on the projection on primitive elements parallel to decomposable elements.

The following example will show that the lack of injectivity of $\varphi$ induces a counter-example in the rigidity theorem.

\begin{example}\label{expleST}
Consider the coassociative associative bialgebras where the mixed distributive law $\lambda$ is the Hopf mixed distributive law.  We compute $\varphi: \mathcal{A}s(V)\to\mathcal{A}s(V)^c$ from the mixed distributive law: it is given by $\varphi(1\ldots n \otimes (v_1, \ldots v_n)) = \sum_{\sigma \in \mathfrak{S}_n} \sigma(1)\ldots\sigma(n) \otimes (v_1, \ldots v_n)$. It is to be noted that the $\varphi$ constructed is not a monomorphism (therefore not an isomorphism and nor a epimorphism), failing to verify all the hypotheses for the existence of a rigidity theorem.

Let us consider Solomon-Tits algebra, see for example \cite{Novelli-Thibon, FNT, Burgunder-Ronco}. Its underlying vector space is the space of surjections ${\mbox {\bf ST}} = \cup_{n,r \geq 1} {\mbox {\bf ST}_n^r}$, where ${\mbox {\bf ST}_n^r}=\{ x : \{1, \ldots, n\} \rightarrow \{1, \ldots, r\}| x \text{ surjective}\}$) (also known as \emph{set compositions}). One can endow this space with a shuffle product and the block coproduct defined as follows. For $x\in {\mbox {\bf ST}_n^r}$, we write $x = (x(1),\dots ,x(n))$ and $r=\max\{x(i),1\leq i\leq n\}$. Let $x\in {\mbox {\bf ST}_n^r}$, $y\in {\mbox {\bf ST}_m^s}$. The (vertical) shuffle product $\star$ is defined by:
\begin{equation}
 x\star y= \sum f\circ (x(1),\cdots x(n),y(1)+r,\cdots y(m)+r), 
\end{equation}
where the sum is over all the stuffles $f\in SH(r,s)$, i.e. any bijective map $f : \{1, \ldots, r+s\} \rightarrow \{1, \ldots, r+s\}$ satisfying $f(1)<\ldots <f(r)$, $f(r+1)<\ldots <f(r+s)$. 

Given a map $x:[n]\longrightarrow {\mathbb N}$ there exists a unique surjective map ${\mbox {std}(x)}$ in $ {\mbox {\bf ST}}$ such that $x(i) <x (j)$ if, and only if, ${\mbox {std}(x)}(i) < {\mbox {std}(x)}(j)$, for $1\leq i,j\leq n$. The map ${\mbox {std}(x)}$ is called the {\it standardisation} of $x$ (see for instance \cite{Novelli-Thibon}).  Let $K=\{ j_1<\dots <j_l\}\subseteq \{1,\dots ,r\}$, we define the co-restriction of $x$ to $K$ by $x\vert ^K:={\mbox {std}(x(s_1),\dots ,x(s_q))}$, for $x^{-1}(K)=\{ s_1<\dots <s_q\}$. We define the block coproduct $\Delta_{\textrm{block}}(x)$  by:
\vspace{-0.2cm}
\begin{equation*}
\Delta_{\textrm{block}} (x) = \sum _{i=1}^{r-1} x\vert ^{\{1,\dots , i\}}\otimes x\vert ^{\{ i+1, \dots , r\}},
\end{equation*}
and we extend it by linearity to all $\mathbb K[{\mbox {\bf ST}}]$. 

\begin{examples} Some examples of products and coproducts are given by:
\begin{equation*}
1 \star 1 = 12+21, \ 11 \star 1 = 112+221, \ 12 \star 1 = 123+132+231
\end{equation*}
\begin{equation*}
213 \star 11 = 21344+32411+31422+21433,
\end{equation*}
\begin{equation*}
\Delta_{\textrm{block}}(2433142) = \epsilon \otimes 2433142 + 1 \otimes 132231 + 212 \otimes 2112 + 23312 \otimes 11 + 2433142 \otimes \epsilon
\end{equation*}
where $\epsilon$ is the empty word (the unique surjection from a set of cardinality $0$ to itself).
\end{examples}

The mixed distributive law between the product and the coproduct is the Hopf mixed distributive law, which can be checked by direct inspection. 

 ${\mbox {\bf ST}}$ is conilpotent as for any element in $x\in{\mbox {\bf ST}_n^r}$ $\Delta_{\textrm{block}}^{r+1}(x)=0$, where $\Delta_{\textrm{block}}^2=\Delta_{\textrm{block}}$ and $\Delta_{\textrm{block}}^k=(\Delta_{\textrm{block}}\otimes {\mbox {id}}^{\otimes k-1})\circ \Delta_{\textrm{block}}^{k-1}$.
 
 By a direct computation, one can prove that the element $112$ cannot be constructed as a linear combination of products of primitive elements. And ${\mbox {\bf ST}}$ is therefore not generated by its primitives though it does verify the conilpotent and the existence of the mixed distributive law hypotheses. 
\end{example}

\subsection{Homogeneous confluence laws.} \label{hom}

We first decribe the link between filtration and cofiltration in the free bialgebra. We then use this link to define homogeneous confluence laws, which are the step between the isomorphism $\varphi$ and the associated confluence law $\alpha$.

\begin{proposition}
The data of a family of isomorphisms $\varphi_V: \A(V) \rightarrow \C^c(V)$ functorial in $V$ and sending natural graduation by product to natural cograduation by coproduct, is equivalent to the data of a family of isomorphisms of $\mathfrak{S}$-modules $\varphi_n: \A(n) \rightarrow \C^*(n)$. 
\end{proposition}

\begin{proof}
 As $\K$ is a field of characteristic $0$, we can identify coinvariants and invariants and the notion of free (conilpotent) coalgebra over the operad $\C$ is given by the Schur functor defined by $\C^c(V)=\bigoplus \C^*(n) \otimes_{\mathfrak{S}_n} V^{\otimes n} $. 

The data of a family of isomorphisms of $\mathfrak{S}$-modules $\varphi_n: \A(n) \rightarrow \C^*(n)$ defines a natural transformation between the Schur functor associated to the operad $\mathcal{A}$, $\A: V \mapsto \bigotimes_{n \geq 1} \mathcal{A}(n) \otimes V^{\otimes n}$ and the Schur functor associated to the cooperad $\mathcal{C}^*$, $\C^c$  described above. This natural transformation is directly equivalent to a family of isomorphisms $\varphi_V: \A(V) \rightarrow \C^c(V)$ functorial in $V$ sending graduation $\mathcal{A}(n) \otimes V^{\otimes n}$ to cograduation $\C^*(n) \otimes_{\mathfrak{S}_n} V^{\otimes n} $. 

Let us now suppose the existence of a family of isomorphisms $\varphi_V: \A(V) \rightarrow \C^c(V)$ functorial in $V$, sending graduation by product to cograduation by coproduct. Then, $\varphi_V(\A(n) \otimes_{\mathfrak{S}_n} V^{\otimes n}) = \C^ *(n) \otimes_{\mathfrak{S}_n} V^{\otimes n}$, for any integer $n$. Hence $\varphi_V$ defines a family of isomorphisms of $\mathfrak{S}$-modules $\varphi_n: \A(n) \rightarrow \C^*(n)$. 
\end{proof}

We will denote by $\varphi: \A \rightarrow \C^*$ the morphism of $\mathfrak{S}$-modules induced by the family of $\varphi_n$.

\begin{remark}
For the trivial representation, the only isomorphisms $\varphi : \operatorname{Comm}(V) \rightarrow\operatorname{Comm}^c(V)$ are homotheties, but it is not always the case (it depends on the chosen decomposition in irreducible representations of the $\mathfrak{S}$-module, which is unique only up to isomorphisms).

\end{remark}

We call \emph{homogeneous confluence law} the set of all rewritings of compositions of a cooperation of arity $n$, $\delta\in \C(n)$, with an operation of arity $n$, $\mu \in \A(n)$, applied to primitive elements in terms of the primitive elements $p_i$. We denote this set by $(\alpha_h)$. Then we have: 
\begin{equation}
\begin{split}
(\alpha_h)=\{\delta \circ \mu (p_1, \ldots, p_n) =& \sum_{\sigma \in \mathfrak{S}_n} b^ \sigma_{\delta, \mu} p_{\sigma(1)} \otimes \ldots \otimes p_{\sigma(n)} \\ &| \delta \in \C(n), \mu \in \A(n), b^ \sigma_{\delta, \mu} \in \mathbb{K},  n \geq 1, p_i \in \mathcal{F}_1\mathcal{H}\}.
\end{split}
\end{equation}

In other words, homogeneous confluence laws are to confluence laws what filtered distributive laws are to distributive laws (see \cite{FiltDistLaw, PoissBrack}).

A homogeneous confluence law is then strictly equivalent to endowing the free $\A$-algebra over a vector space $V$ with a structure of $\C$-coalgebra sending the graduation of the algebra to the cograduation of the coalgebra. Indeed, the following result follows:
\begin{lemma}\label{phiLn}
The data of $\varphi$ is equivalent to the data of a homogeneous confluence law.
\end{lemma}

\begin{proof}
\begin{itemize}
\item[$\Leftarrow$] Given a homogeneous confluence law, we can consider the map from $\A(V)$ to itself where at the source $\A(V)$ is viewed as an algebra and at the target $\A(V)$ is seen as a coalgebra. It produces  a family of morphism $\psi_V$, sending filtration by product to cofiltration by coproduct, and as result a morphism $\psi_n$ defined by 
\begin{equation*}
\psi_n (\mu) (\delta)= b^{id}_{\delta, \mu}. 
\end{equation*} 
\item[$\Rightarrow$]Denoting by $\left\langle , \right\rangle$ the duality pairing between $\C(n)$ and its dual space $\C^*(n)$, and $x^{\sigma}$ the action of an element $\sigma$ of the symmetric group on an element $x$ of $\C(n)=\A(n)$, we define the homogeneous confluence law $(\alpha_p)$ from the morphism $\varphi$ by:
\begin{equation} \label{crochet}
\delta \circ \mu (p_1, \ldots, p_n) = \sum_{\sigma \in \mathfrak{S}_n}\left\langle  \delta^{\sigma}, \varphi_n(\mu)\right\rangle p_{\sigma(1)} \otimes \ldots \otimes p_{\sigma(n)},
\end{equation} 
where $\delta \in \C(n)$, $\mu \in A(n)$ and $p_i \in V$.

Notice that we have: 
\begin{equation*}
\left\langle   \delta^{\sigma}, \varphi_n(\mu)\right\rangle = \left\langle  \delta, \varphi_n(   \mu^{\sigma^{-1}})\right\rangle.
\end{equation*}
\end{itemize}

The two constructions are inverse of each other.

\end{proof}

\subsection{Equivalence between conditions.} 
\label{Equiv}

Now, we will use homogeneous confluence laws and the previous subsection to prove the equivalence between hypothesises of the rigidity theorem, as stated by Loday.

\begin{proposition}\label{mdl}
Let $\mathcal C$ and $\mathcal A$ be two algebraic operads.
The existence of a confluence law $\alpha$  implies the existence of a family of $\mathfrak{S}_n$-module morphism $\varphi_n:\A(n)\to \C^*(n)$, for any positive integer $n$. Moreover, the bijectivity of all $\varphi_n$ enables us to define the associated confluence law.
\end{proposition}

\begin{remark}
If such bijections exist, then $\A$ and $\C$ have the same underlying $\mathfrak{S}_n$-module.
\end{remark}

\begin{proof} 

\begin{itemize}
\item[$ (\alpha) \Rightarrow \varphi$:] An homogeneous confluence law can be directly computed from a confluence law $(\alpha)$ by forgetting some information.

\item[ $\varphi \Rightarrow (\alpha)$:] 
We define from $\varphi$ the associated confluence law $(\alpha)$.
\end{itemize}
We use the following notation $M_n:=\mathcal{A}(n) \otimes_{\mathfrak{S}_n} V^{\otimes n}$ and $H_n:=\varphi(M_n)=\mathcal{C}^*(n) \otimes_{\mathfrak{S}_n} V^{\otimes n}$.

The confluence law is defined as the following composite, using notations of Definition \ref{defcog}:
\begin{align*}
\mathcal{C}(k) \otimes \mathcal{A}(n) &\otimes_{\mathfrak{S}_n}  V^{\otimes n} \xrightarrow{id \otimes \varphi \otimes id^{\otimes n}} \mathcal{C}(k) \otimes \mathcal{C}^*(n) \otimes_{\mathfrak{S}_n} V^{\otimes n} = \mathcal{C}(k) \otimes H_n \\
&\xrightarrow{\gamma^k_\C} 
\sum_{\substack{l_1+\ldots+l_k=n,\\ l_i \geq 1}} \bigotimes_{i=1}^k H_{l_i}  \xrightarrow{(\varphi^{-1})^{\otimes k}} \sum_{\substack{l_1+\ldots+l_k=n,\\ l_i \geq 1}} \bigotimes_{i=1}^k M_{l_i} \\&= \sum_{\substack{l_1+\ldots+l_k=n,\\ l_i \geq 1}} \left( \bigotimes_{i=1}^k \mathcal{A}(l_i) \right) \otimes_{\mathfrak{S}_n} V^{\otimes n}.
\end{align*}

As the operations are functorial in $V$, we get a family of maps $\alpha_{m,n} : \C(m) \otimes \A(n) \rightarrow \A^{\otimes m}(n)$. Moreover, $\varphi$ is a morphism of $\mathfrak{S}$-modules, hence maps $\alpha_{m,n}$ are compatible with the action of the symmetric group $\mathfrak{S}_m \times \mathfrak{S}_n$.

We finally have to show the compatibility with the operad structure of $\C$. This comes from the following commuting diagram, holding for any vector space $V$:

\hspace{-1cm}
\begin{tikzpicture}
  \matrix (m) [matrix of math nodes,row sep=3em,column sep=2em,minimum width=2em] {
     \left(\C(k_1)\otimes \ldots \otimes \C(k_l)\right) \otimes \C(l) \otimes \A(n) \otimes V^{\otimes n} &\C(k) \otimes \A(n) \otimes V^{\otimes n} \\
     \sum_{\sum_{j=1}^l q_j=n} \bigotimes_{j=1}^l \left( \C(k_j)\otimes \A(q_j) \otimes V^{\otimes q_j}\right)   & \sum_{\sum_{i=1}^k r_i = n} \bigotimes_{i=1}^k \left(\A (r_i)\otimes V^{\otimes r_i} \right),
   \\};
  \path[-stealth]
    (m-2-1) edge node [below] {$(\alpha)^{\otimes l}$} (m-2-2)
        (m-1-1)    edge node [above] {$\mu_\C \otimes id$} (m-1-2)
    (m-1-1) edge node [left] {$(id)^{\otimes l} \otimes \alpha$} (m-2-1)
    (m-1-2) edge node [right] {$\alpha$} (m-2-2);
\end{tikzpicture}

where $k = k_1 + \ldots + k_l$.

This diagram is obtained by using the definition of $\alpha$ as the composition given above. Using this definition, the commutativity of the diagram relies on two facts: most of the operations do not occur on the same part of the tensorial product and the definition of the coproduct gives itself the commutativity of one part of the diagram thanks to the following diagram, holding on any of the previous $H_n$:

\begin{tikzpicture}
  \matrix (m) [matrix of math nodes,row sep=3em,column sep=2em,minimum width=2em] {
     \left(\C(k_1)\otimes \ldots \otimes \C(k_l)\right) \otimes \C(l) \otimes H_n &\C(k) \otimes H_n  \\
     \sum_{\sum_{j=1}^l q_j=n} \bigotimes_{j=1}^l \left( \C(k_j)\otimes H_{q_j} \right)  & \sum_{\sum_{i=1}^k r_i = n} \bigotimes_{i=1}^k H_{r_i},
   \\};
  \path[-stealth]
    (m-2-1) edge node [below] {$(\alpha)^{\otimes l}$} (m-2-2)
        (m-1-1)    edge node [above] {$\mu_\C \otimes id$} (m-1-2)
    (m-1-1) edge node [left] {$(id)^{\otimes l} \otimes \gamma^l_\C$} (m-2-1)
    (m-1-2) edge node [right] {$\gamma^{k_1}_\C \otimes \ldots \otimes \gamma^{k_l}_\C$} (m-2-2);
\end{tikzpicture}

where $k = k_1 + \ldots + k_l$.
\end{proof}

\begin{remark} \begin{itemize}
\item  If the confluence law $\alpha$ is given by an isomorphism $\varphi$, we will write equivalently the associated bialgebras $\C^c-_{\varphi}\A$-bialgebras or $\C^c-_{\alpha}\A$-bialgebras.
\item There can be different confluence laws associated to the same couple of operads: see for instance the mixed distributive laws for the Dendriform coDendriform bialgebra computed by Foissy in \cite{Foissy} and the one referred in \cite{DendTridend}.
\end{itemize}
\end{remark}

\emph{Projection}

The last step before reaching the rigidity theorem is to prove the existence of a "good" projection of the coalgebra on the primitive elements, named \emph{idempotent}.

Let us now consider a (not necessarily free or cofree) conilpotent $\C^c-_{\alpha}\A$-bialgebra $\mathcal{H}$. We denote by $\mathcal{F}_n$ the cofiltration on $\mathcal{H}$ and $\varphi$ the $\mathfrak{S}$-module morphism associated to the confluence law.

We first show the following lemmas: 
\begin{lemma}\label{LemGen}
The bialgebra $\mathcal{H}$ is generated, as an algebra, by its primitive elements.
\end{lemma}

\begin{proof} We proove the lemma by reductio ad absurdum. 
Let us consider the minimal integer $n \geq 2$ such that there exists an element $x$ in $\F_n\mathcal{H}$, which is not in the subalgebra generated by primitive elements of $\mathcal{H}$. We will construct a $y$ in the subalgebra generated by primitive elements of $\mathcal{H}$ such that $x-y$ belongs to $\F_{n-1}\mathcal{H}$, and then is in the subalgebra generated by primitive elements.

For any (linear) basis $(\mu_1, \ldots, \mu_k)$ of $\A(n)$, there exists a basis $(\delta_1, \ldots, \delta_k)$ of $\C(n)$ such that $\varphi(\mu_i) = \delta_i^*$. We denote by $a_{i,j}^\sigma$ the real number $\delta_i^*(\delta_j^{\sigma})$, denoting by $\delta_j^{\sigma}$ the action of an element $\sigma$ of the symmetric group on an element $\delta_j$ of the considered $\mathfrak{S}$-module. The bialgebra satisfies a confluence law and then also a homogeneous confluence law. Hence, we have for any primitive elements $p_1, \ldots, p_n$, according to Equation \eqref{crochet}:
\begin{equation}\label{abv}
\delta_j \circ \mu_i(p_1, \ldots, p_n) = \sum_{\sigma \in \mathfrak{S}_n} a_{i,j}^\sigma p_{\sigma(1)} \otimes \ldots \otimes p_{\sigma(n)}. 
\end{equation}
As $x$ belongs to $\F_n\Hc$, for any $j$ between $1$ and $k$ the coproduct $\delta_j$ of the element $x$ can be written as:
\begin{equation}
\delta_j(x) =  \sum_{(p_1, \ldots, p_n) \in \F_1\Hc } \sum_{\tau \in \mathfrak{S}_n} d^j_{\tau} p_{\tau(1)} \otimes \ldots \otimes p_{\tau(n)}.
\end{equation}

Without loss of generality, we can consider the case:
 
\begin{equation} \label{x}
\delta_j(x)=\sum_{\tau \in \mathfrak{S}_n} d^j_{\tau} p_{\tau(1)} \otimes \ldots \otimes p_{\tau(n)}
\end{equation}
for a given tuple of primitive elements $(p_1, \ldots, p_n)$, as $x$ can be decomposed as a sum of such elements.

Let us now consider an element $y = \sum_{i=1}^k \sum_{\tau \in \mathfrak{S}_n} c_\tau^i \mu_i(p_{\tau(1)}, \ldots, p_{\tau(n)})$. Let us determine the appropriate $c_\tau^i$. Then, we have, using Equation \eqref{abv}:
\begin{equation} \label{y}
\delta_j(y) = \sum_{i=1}^k \sum_{\tau, \sigma \in \mathfrak{S}_n} c_\tau^i a_{i,j}^\sigma p_{\tau \circ \sigma(1)} \otimes \ldots \otimes p_{\tau \circ \sigma(n)}.
\end{equation}
Then, using Equations \eqref{x} and \eqref{y}, for all $j \in \{1, \ldots, k\}$, $\delta_j(x) = \delta_j(y)$ is equivalent to:
\begin{equation} \label{r1}
d^j_{\tau} = \sum_{i=1}^k \sum_{\sigma \in \mathfrak{S}_n} c^i_{\tau \circ \sigma^{-1}} a_{i,j}^\sigma.
\end{equation}

Moreover, for any $j$, $\delta_j^{\sigma} = \sum_{i=1}^k  a_{i,j}^\sigma \delta_i$, which gives on $x$, for any $\tau, \sigma \in \mathfrak{S}_n$:
\begin{equation} \label{r2}
d^j_\tau = \sum_{i=1}^k a_{i,j}^\sigma d^i_{\tau \circ \sigma^{-1}}
\end{equation}

Thanks to Equations \eqref{r1} and \eqref{r2}, choosing $c^i_\sigma = \frac{1}{n!} d^i_{\sigma}$ for any $i$ and $\sigma$ gives $y$ in the subalgebra of $\mathcal{H}$ generated by primitive elements such that $x-y$ is in $\F_{n-1}\mathcal{H}$, hence $x-y$ is also in the subalgebra of $\mathcal{H}$ generated by primitive elements by minimality of $n$ and so is $x$.
\end{proof}

\begin{lemma} \label{decomp} The vector space $\mathcal{F}_{n-1}$ admits a supplementary space in $\mathcal{F}_n$, which is $M_n=\{\mu(p_1, \ldots, p_n)| \mu \in \A(n), p_i \in \F_1\}$. 
\end{lemma}

\begin{proof}
\begin{itemize}
\item First, we have by definition $\F_{n-1} \subseteq \F_{n}$.

\item Let us now show that $M_n \cap \F_{n-1} = \{0\}$, for all $n$.  Let us consider an element $\mu(p_1, \ldots, p_n)$ in $M_n \setminus \{0\}$. As $\varphi$ is injective, there is an element $\delta \in \C(n)$ such that $\delta \circ \mu (p_1, \ldots, p_n) \neq 0$, thus $\mu(p_1, \ldots, p_n) \notin \F_{n-1}$.

\item Then, we have $M_n \subset \F_n$, for all $n$. Thus, $\F_n \supset \F_{n-1} \oplus M_n$. This comes from Equation \eqref{crochet}.

\item Let us prove that $\F_n \subset \F_{n-1} \oplus M_n$. As an algebra, $\mathcal{H}=\A(V)/(R)$. By Lemma \ref{LemGen}, we have that $V \subseteq \F_1$. Let us consider an element $x \in \mathcal{F}_n$. Then we can write (not necessarily uniquely) $x$ as:
\begin{equation*}
x = \sum_{j \geq 1} \mu_j(p^j_1, \ldots, p^j_j),
\end{equation*}
 where $p^j_k \in \F_1$ and $\mu_j \in\A(j)$.
 
 From Equation \eqref{crochet}, we have $M_j \subset \F_j$. From the bijectivity of $\varphi$, $M_j \cap \F_{j-1} = \{0\}$ (see the second point of this proof). Then, as $x \in \mathcal{F}_n$, we get the decomposition:
 \begin{equation*}
x = y + z,
\end{equation*}
 where $y \in \F_{n-1}$ and $z \in M_n$.
\end{itemize}
\end{proof}

For any element $x \in \mathcal{F}_n$, we can then define the projection $e$ on $\mathcal{F}_1 = M_1$ parallel to the $M_j$, $j>1$.

The map $e$ is linear and satisfies $e \circ e = e$ and $\Im(e) = \mathcal{F}_1 = \Prim(\mathcal{H})$. We will call this map the \emph{idempotent} associated to $\C^c-_{\alpha}\A$-bialgebras (see Corollary \ref{unique}).

\subsection{Rigidity theorem.}\label{rig}
We now give a new formulation of rigidity theorem which takes into account the previous subsections and confluence laws.

\begin{theorem} \label{thmrig} Let $\K$ be a field of characteristic $0$ and let us consider two connected algebraic operads $\A$ and $\C$, such that $\A(n)$ and $\C(n)$ are finite dimensional vector spaces.
To any family of $\mathfrak{S}_n$-modules isomorphisms $\varphi_n: \A(n) \rightarrow \C^*(n) $ can be associated a confluence law $(\alpha)$ such that any conilpotent $\C^c-_{\alpha}\A$-bialgebras is free and cofree over the vector space of its primitive elements
\begin{equation*}
\A(\Prim \mathcal{H})\cong \mathcal{H} \cong \C^c(\Prim \mathcal{H}).
\end{equation*}
Moreover, if any free and cofree $\C^c-_{\alpha}\A$-bialgebra is also a $\C^c-_{\tilde{\alpha}}\A$-bialgebra, with $\tilde{\alpha}$ a confluence law or a mixed distributive laws, then any conilpotent $\C^c-_{\tilde{\alpha}}\A$-bialgebras is free and cofree over its primitive elements. In other words, the result does not depend on the choice of a confluence law.
\end{theorem}

Then, on any $\C^c-_\alpha \A$-bialgebra the idempotent is exactly the projection on $\F_1$ parallel to the $M_n$ (see \ref{decomp}), hence the following result:
\begin{corollary} \label{unique}
The idempotent $e$ is unique.
\end{corollary}

The idempotent described above is the same as the one introduced in \cite{GBO}. It is the generalisation of the Eulerian idempotent used in \cite{Patras}, whose origin lays in \cite{Chen}, obtained in the case $\A=\C=\comm$.

We follow here the sketch of the proof of Loday in \cite{GBO} and Patras in \cite{Patras}, using the confluence laws and the idempotent introduced previously. 
\begin{proof} 
We want now to prove that there is an isomorphism between $\mathcal H$ and $\mathcal \A(\mbox{Prim\ }\mathcal H)$, where the latter is the free and cofree bialgebra over $\prim \Hc$. We call $i$ and $p$ respectively the canonical injection from $\mbox{Prim\ }\mathcal{H}$ to $\A(\prim \Hc)$ and surjection from $\C^c(\prim \Hc)$ to $\prim \Hc$.
Let us define $\tilde e:\mathcal H\to  \mathcal \C^c(\mbox{Prim\ }\mathcal H)$ the unique lifting of $e$ by the universal property of the $\C$-conilpotent cofreeness. Consider $\iota:\mbox{Prim\ }\mathcal H\to \mathcal H$ the natural injection and $\tilde\iota:\mathcal \A(\mbox{Prim\ }\mathcal H)\to\mathcal H$ its lifting by the universal property of $\mathcal P$ freeness. We represent below these morphisms:
\begin{center}

\begin{tikzpicture}
  \matrix (m) [matrix of math nodes,row sep=3em,column sep=4em,minimum width=2em] {
     \A(\prim \Hc) &  \\
     \Hc & \prim \Hc \\
     \C^c(\prim \Hc)& \\
      \\};
  \path[-stealth]
    (m-1-1) edge node [left] {$\tilde \iota$} (m-2-1)
    (m-2-2.north) edge node [above] {$i$} (m-1-1.south east)
    (m-3-1.north east) edge node [below] {$p$} (m-2-2.south)
     (m-2-1.south east)edge node [below] {$e$} (m-2-2.south west)
     (m-2-2.north west) edge node [above] {$\iota$} (m-2-1.north east)
    (m-2-1) edge node [left] {$\tilde e$} (m-3-1)
    (m-1-1.south west) edge[bend right] node [left] {$\varphi$} (m-3-1.north west);
\end{tikzpicture}
\end{center}
\vspace{-1.2cm}

First, $\tilde\iota$ is a bialgebra morphism: it is an algebra morphism by the universal property and we show using the filtration of $\mathcal H$ provided by the conilpotency of $\mathcal H$ that it is also a coalgebra morphism. Indeed, by construction, $\F_1 \A(\prim \Hc) = \prim \Hc$ is sent bijectively to $\F_1 \Hc=\prim \Hc$. Let us now consider $x\in \F_{n} \A(\prim \Hc)$. By freeness, there exists $p_1, \ldots, p_n$ in $\F_1 \A(\prim \Hc)$ and $\mu \in \A(n)$ such that $x = \mu(p_1, \ldots, p_n)$. Hence, we have for any cooperation $\delta \in \C(k)$:
\begin{equation*}
(\tilde{\iota} \otimes \ldots \otimes \tilde{\iota}) \circ \delta(x) = (\tilde{\iota} \otimes \ldots \otimes \tilde{\iota}) \circ \delta \circ \mu (p_1, \ldots, p_n).
\end{equation*}
The confluence law can then be denoted, by distributivity, by: 
\begin{equation*}
\delta \circ \mu (p_1, \ldots, p_n)= \sum_{i=1}^l \mu^i_1 \left(p_{\sigma_i(1)}, \ldots, p_{\sigma_i(r^i_1)} \right) \otimes \ldots \otimes \mu^i_k  \left(p_{\sigma_i(n-r^i_k+1)}, \ldots, p_{\sigma_i(n)}\right),
\end{equation*}
for any primitive elements $p_1, \ldots, p_n$, where $\mu^i_j$ has arity $r^i_j$, $\sum_{j=1}^k r^i_j=n$ and  $\sigma_i$ is in $\mathfrak{S}_{\textrm{shuffle}}$, the set of permutations of $\{1, \ldots, n\}$ such that $\sigma_i(\sum_{j=1}^l r^i_j+1)< \ldots < \sigma_i(\sum_{j=1}^l r^i_j+r^i_{l+1})$ for any $0 \leq l \leq k-1$.

Note that the reasoning does not depend on the choice of such a confluence law.

We have:
\begin{align*}
(\tilde{\iota} \otimes \ldots \otimes \tilde{\iota}) \circ \delta(x) &= (\tilde{\iota} \otimes \ldots \otimes \tilde{\iota}) \circ \delta \circ \mu (p_1, \ldots, p_n) \\
 &=  \sum_{i=1}^l \tilde{\iota} \circ \mu^i_1 \left(p_{\sigma_i(1)}, \ldots, p_{\sigma_i(r^i_1)} \right) \otimes \ldots \otimes \tilde{\iota} \circ \mu^i_k  \left(p_{\sigma_i(n-r^i_k+1)}, \ldots, p_{\sigma_i(n)}\right) \\
 &= \sum_{i=1}^l \bigotimes_{j=1}^k \mu^i_j \left(\tilde{\iota}(p_{\sigma_i(r^i_j+1)}), \ldots, \tilde{\iota}(p_{\sigma_i(\sum_{j=1}^l r^i_j+r^i_{l+1})}) \right)  \text{ by algebra morphism} \\
 &= \delta \circ \mu (\tilde{\iota}(p_1), \ldots, \tilde{\iota}(p_n)) \text{with the confluence law} \\
 &= \delta \circ \tilde{\iota} \circ \mu (p_1, \ldots, p_n) \text{ by algebra morphism} \\
 &= \delta \circ \tilde{\iota}(x).
\end{align*}
This proves that $\tilde{\iota}$ is a bialgebra isomorphism.

\vspace{0.5cm}

This implies that $\tilde e \circ \tilde{\iota}$ is a coalgebra morphism, which is the identity on primitive elements. Hence, we have by cofreeness the equality $\tilde e \circ \tilde{\iota} = \varphi$, which implies that $\tilde{\iota}$ is injective and $\tilde{e}$ surjective. Finally, $\tilde{e}$ is also injective: otherwise, there would be a minimal integer $m$ ($m>1$) such that there exists $y \in \F_m \Hc$ such that $\tilde{e}(y)=0$. By definition of $\F_m \Hc$, there exists a cooperation $\delta \in \C(m)$ such that $\delta(y) \neq 0$. Using the fact that $e$ is a coalgebra morphism, we obtain $\delta \circ \tilde{e} (y)= (\tilde{e} \otimes \ldots \otimes \tilde{e}) \circ \delta (y)= 0$ and $\delta(y) \neq 0$ which contradicts the minimality of $m$.

Finally, the map $\tilde e: \Hc \to \C^c(\prim \Hc)$  is a vector space isomorphism which preserves the graduation and the result follows using the isomorphim $\varphi$ between $\A(\prim \Hc)$ and $\C^c(\prim \Hc)$.
\end{proof}

An immediate result of this theorem is the following corollary:

\begin{corollary} Given two algebraic operads $\C$ and $\A$ and a confluence law $\alpha$ induced by an isomorphism $\varphi: \A \rightarrow \C^*$, an $\A$-algebra $\mathcal{H}$ is free if and only if it is possible to define a $\C^c$-coproduct on $\mathcal{H}$ satisfying the confluence law $\alpha$.
\end{corollary}

\begin{proof}
\begin{description}
\item[$\Rightarrow$] This is a consequence of the existence of a confluence law. If $\Hc$ is free, then there exists a vector space $V$ such that $\Hc=\A(V)$ and then $\Hc$ is naturally equipped with a structure of $\C^c-\A$-bialgebras thanks to the mixed distributive law $\alpha$.

\item[$\Leftarrow$] This is a consequence of the rigidity theorem.
\end{description}
\end{proof}

This provides an efficient tool to solve the problem of freeness of algebras encoded by an operad (see \cite{EPLM}) and to give generators and divisibility in free algebras.

Moreover, another application of the theorem is the following answer to a conjecture of Loday in \cite{GBO}:

\begin{theorem}
For any connected algebraic operads $\A$ and $\C$ sharing the same underlying finite dimensional $\mathfrak{S}_n$-modules $\A(n)=\C(n)$, there exists a family of $\mathfrak{S}_n$-modules isomorphisms $\varphi_n: \A(n) \to \C^*(n)$, and thus there exists a confluence law $\alpha$ associated to $\varphi$ such that any conilpotent $\C^c-_{\alpha}\A$-bialgebra is free and cofree over its primitives.
\end{theorem}

\begin{proof}
Any character of a representation of the symmetric group is integer-valued, hence for any representation $W$ of the symmetric group, it is possible to construct a $\mathfrak{S}_n$-module isomorphism between $W$ and its dual $W^*$. Let us choose $W = \A(n)$: we obtain a family of $\mathfrak{S}_n$-modules isomorphisms between $\A(n)$ and $\C^*(n) = \A^*(n)$ satisfying the hypothesis of Theorem \ref{thmrig}. Applying this theorem gives the result.
\end{proof}

The following corollary is then immediate:
\begin{corollary}
If $\A$ and $\C$ are two operads sharing the same underlying $\mathfrak{S}_n$-module, then any free $\A$-algebra is a free $\C$-algebra.
\end{corollary}

\begin{remark}
Note that it is not the case for non free algebras.
\end{remark}

\subsection{Orthogonal bases.} \label{orthbas}

In this subsection, we provide an explicit inductive description of confluence laws and idempotent in favourable cases. 

For every $n \geq 1$, we consider, if they exist, two generating sets $B\A_n=(a^n_1,  \ldots, a^n_{i_n})$ and $B\C_n=(c^{n}_1,  \ldots, c^{n}_{i_n})$ of the $\mathfrak{S}_n$-module $\A(n)$ and $\C(n)$ respectively, such that $< \left(c^n_i\right)^{\sigma}, \varphi(a^n_j)>=0$ for all $\sigma \in \mathfrak{S}_n$ and $i \neq j$. Let us remark that to obtain such generating sets, we should take at most one element in every orbit of $\mathfrak{S}_n$ on $\A(n)$. Such a generating set can be found for instance if the only relations in the $\mathfrak{S}_n$-module $\A(n)$ are of the form $x = x^{\sigma}$ for $x$ in $\A(n)$ and $\sigma \in \mathfrak{S}_n$.

Then, considering such a basis, Equation \eqref{crochet} gives:
\begin{align*}
c^n_j \circ a^n_i (p_1, \ldots, p_n) &= \sum_{\sigma \in \mathfrak{S}_n}\left\langle {( c^n_j)}^{\sigma}, \varphi_n(a^n_i)\right\rangle p_{\sigma(1)} \otimes \ldots \otimes p_{\sigma(n)} \\
&= \sum_{\sigma \in \mathfrak{S}_n} \left\langle  {(c^n_j)}^{\sigma}, (c^n_i)^*\right\rangle p_{\sigma(1)} \otimes \ldots \otimes p_{\sigma(n)} \\
&= \sum_{\sigma \in \operatorname{Aut}(c^n_j)} \delta_{i,j} p_{\sigma(1)} \otimes \ldots \otimes p_{\sigma(n)}.
\end{align*}
with $\operatorname{Aut}(c^n_j)=\{\sigma \in \mathfrak{S}_n, (c^n_j)^\sigma = c^n_j\}$.

Let us remark that we have $\operatorname{Aut}(c^n_j)  \supseteq \operatorname{Aut}(a^n_j)$ by definition of $\varphi$, with the equality when $\varphi$ is injective, which will be supposed in what follows. 

The confluence law is then given for any cooperation $\delta \in \C(k)$ and operation $\mu \in \A(n)$ by induction on the cofiltration $\F_{\lambda_1} \otimes \ldots \otimes \F_{\lambda_n}$, with $\sum_{i=1}^{n} \lambda_i = k$ by:
\begin{equation*}
\begin{split}
\alpha_{(\delta,\mu,k)}:= \sum_{\substack{(a^{l_1}_{j_1}, \ldots, a^{l_n}_{j_n}) \\ \in B\A_{\lambda_1} \times \ldots \times B\A_{\lambda_n}}} \frac{1}{\prod_i|\operatorname{Aut}(a_i^{l_i})|} &\left(  T(\otimes_i a_{j_i}^{l_i})- T_{k-1}(\otimes_i a_{j_i}^{l_i}) \right) \\ &\circ \left(c_{j_1}^{l_1} \otimes \ldots \otimes c_{j_n}^{l_n}\right),
\end{split}
\end{equation*}
where $T(\otimes_i a_{j_i}^{l_i})$ is the tensorial product of operations obtained when evaluating $\delta \circ \mu \circ \left(a_{j_1}^{l_1} \otimes \ldots \otimes a_{j_n}^{l_n}\right)$ on primitive elements (see \ref{mdl}) and $T_k(\otimes_i a_{j_i}^{l_i})$ is the tensorial product of operations obtained when evaluating $\alpha_{(\delta,\mu,k)} \circ \left(a_{j_1}^{l_1} \otimes \ldots \otimes a_{j_n}^{l_n}\right)$ on primitive elements.

The maps $\alpha_{(\delta,\mu,k)}$ give the mixed distributive law on elements in $\F_{\lambda_1}\otimes \ldots \otimes \F_{\lambda_n}$, where $\sum_i \lambda_i = k$.

The idempotent is given as the inductive limit of the maps $e_n: \mathcal{H} \rightarrow \mathcal{H}$ defined on the cofiltration $\F_n$ by $e_1=\operatorname{id}$ on $\F_1$ and on $\F_n$ by:
\begin{equation*}
e_{n} = e_{n-1}-\sum_{i} \frac{1}{|Aut(a^n_i)|} e_{n-1} \circ a^n_i \circ c_i^n.
\end{equation*}

Note that by decomposition of Lemma \ref{decomp}, $e_j=e_n$ on $\F_n$ for all $j \geq n$, hence $e$ is the identity on $\F_1$. We prove in what follows that the map obtained is an idempotent:

\begin{lemma} \label{De}
The constructed map satisfies:
$\Delta \circ e = 0$ on any conilpotent $\C^c-_{\varphi} \A$ bialgebra $\Hc$.
\end{lemma}

\begin{proof}[Proof of Lemma \ref{De}]
We prove by induction on $n$ that: 
\begin{equation*}
\Delta \circ e_n(x) = 0 \text{ for } x \in \F_n.
\end{equation*}

As $e_1 = \operatorname{id}$ on primitive elements, we have $\Delta \circ e (x) = \Delta (x) = 0$ for any primitive element $x$.
Let us suppose the property true for any $k \leq n$.
If $x \in \F_k$ for $k \leq n$, $\Delta \circ e_{n+1}(x)=\Delta \circ e_{n}(x)=0$ by induction hypothesis.

If $x \in \F_{n+1}\Hc$, we have
\begin{equation*}
\Delta \circ e_{n+1}(x) = \Delta \circ e_n \left( x - \sum_i \frac{1}{|Aut(a^{n+1}_i)|} a^{n+1}_i \circ c_i^{n+1} (x) \right).
\end{equation*}
However, for any $j$, we have:
\begin{equation*}
\begin{split}
c_j^{n+1} \left( x -\sum_i \frac{1}{|Aut(a^{n+1}_i)|} a^{n+1}_i \circ c_i^{n+1} (x) \right) =& c_j^{n+1}(x) - \frac{1}{|Aut(a^{n+1}_i)|} c_j^{n+1} \\ & \circ a_j^{n+1} \circ c_j^{n+1} (x) = 0.
\end{split}
\end{equation*}
Hence $x - \sum_i \frac{1}{|Aut(a^{n+1}_i)|} a^{n+1}_i \circ c_i^{n+1} (x)$ is in $\F_k$ for $k \leq n$ and $\Delta \circ e_{n+1}(x)=0$.
\end{proof}

This provides an explicit inductive description of the idempotent and the confluence law. In the next section, we apply these constructions.

\section{Dual case}

There is no canonical vector space isomorphism between a vector space and its dual. Moreover, as the morphism $\varphi$ must also be a $\mathfrak{S}_n$-module morphism, not all isomorphisms between $\A(n) \simeq \C(n)$ and $\C^*(n)$ will give rise to a rigidity theorem. We provide in this section a criterion on a given basis to get such a family of $\mathfrak{S}_n$-module isomorphisms and then apply it to several cases.

\subsection{General results.} 

We consider two connected algebraic operads $\A$ and $\C$. To satisfy a rigidity theorem, these operads must have the same underlying $\mathfrak{S}$-modules. Let us now consider a basis $B_n$ of the vector space $\A(n) = \C(n)$, we consider the morphism $\varphi_n: \A(n) \to \C^*(n)$ given by the duality with respect to $B_n$. We denote by $B$ the union $\bigcup_{n \geq 1} B_n$.

Let us now consider a free $\A$-algebra $\Hc$. The basis $B$ induces a basis $B_\Hc$ of $\Hc$. Given an element $x$ of $B_\Hc$ and an operation $\mu \in \C(n)$, the coproduct $\Delta_{\mu}$ given by duality on $x$ is then defined as:
\begin{align*}
\Delta_{\mu}(x)& = \sum_{\substack{x_1, \ldots, x_n \in B_{\Hc} \\ \mu(x_1, \ldots, x_n) \ni \lambda.x \\ \lambda \in \mathbb{K}^*}} \frac{1}{\lambda} x_1 \otimes \ldots \otimes x_n \\
& = \sum_{x_1, \ldots, x_n \in B_{\Hc}} \delta_{x^*(\mu(x_1, \ldots, x_n)) \neq 0} \frac{1}{x^*(\mu(x_1, \ldots, x_n))} x_1 \otimes \ldots \otimes x_n,
\end{align*} 
where $\delta_{x^*(\mu(x_1, \ldots, x_n)) \neq 0}$ is the Kronecker symbol.

Note that thanks to Equation \eqref{crochet}, both definitions coincide.

In operadic terms, this definition can be rewritten, for any operation $S \in \A(k)$, cooperation $T \in \C(l)$ and elements $(e_1, \ldots, e_l)$:
\begin{equation} \label{def2}
C^{T^* (S)}_{(e_1, \ldots, e_l)} = \delta_{C^{T(e_1, \ldots, e_l)}_S \neq 0} \frac{1}{C^{T(e_1, \ldots, e_l)}_S},
\end{equation}
where $C^f_x$ denotes the coefficient of $x$ in $f$.

For the $\varphi_n$ to be $\mathfrak{S}_n$-module morphisms, the basis $B$ has to satisfy some conditions:
\begin{definition}

The basis $B$ is said to be a \emph{compatible basis} if products and dual coproducts expressed in this basis commute with the action of the symmetric group. In other words, for any cooperation $\Delta \in \C(n)$, any operation $\mu \in \A(n)$ and any $\sigma \in \mathfrak{S}_n$, we have:
\begin{equation*}
\Delta \circ (  \mu^{\sigma}) = ( \Delta)^{\sigma^{-1}} \circ \mu.
\end{equation*}

\end{definition}

Thanks to the shape of homogeneous confluence laws (Equation \eqref{crochet}), we obtain directly:

\begin{proposition}
If the considered basis $B$ is a compatible basis, then the family $(\varphi_n)$ is a family of $\mathfrak{S}_n$-module isomorphisms and the rigidity theorem applies: any conilpotent $\C^c-_{\varphi}\A$-bialgebra is free and cofree over the vector space of its primitives. \end{proposition}

\begin{remark} An example of bases of operads which are not compatible is given by the Lyndon basis and the comb basis of the operad Lie (see \cite{HoffVesp} for computation on the comb basis).
\end{remark}

If we consider non symmetric operads, any basis is compatible and then we obtain the following result which gives an explicit family of $\mathfrak{S}$-modules isomorphisms $\varphi$:

\begin{corollary}
For any non symmetric operads $\A$ and $\C$, any basis $B_{\A}$ of $\A$ and $B_{\C}$ of $\C$ there exists a confluence law whose associated isomorphisms are given by sending $B_{\A}$ on the dual basis of $B_{\C}$ such that the rigidity theorem applies. 
\end{corollary}

\begin{example} \label{2.1.5} Using the duality on usual bases of known operads, we find back the following cases. We represent a product $\divideontimes$ by \begin{tikzpicture}[scale=0.4,thick]
\draw (-1,1)--(0,0)--(1,1);
\draw (0,0)--(0,-1);
\draw[draw=black,fill=white] (0,0) circle(0.5);
\draw (0,0) node{$\divideontimes$};
\end{tikzpicture} and a coproduct $\Delta_{\divideontimes}$ by \begin{tikzpicture}[scale=0.4,thick]
\draw (-1,-1)--(0,0)--(1,-1);
\draw (0,0)--(0,1);
\draw[draw=black,fill=white] (0,0) circle(0.5);
\draw (0,0) node{$\divideontimes$};
\end{tikzpicture}, omitting to precise the product or the coproduct if there is no ambiguity.
\begin{itemize}
\item \cite{Borel} $\A=$Comm, $\C=$Comm with the Hopf mixed distributive law:

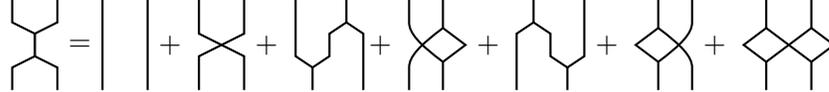
\begin{tikzpicture}[scale=0.3,thick]
\draw (0,0)--(0,1)--(1,1.5)--(1,2.5)--(0,3)--(0,4);
\draw (2,0)--(2,1)--(1,1.5)--(1,2.5)--(2,3)--(2,4);
\draw (3,2) node{$=$};
\draw (4,0)--(4,4);
\draw (6,0)--(6,4);
\draw (7,2) node{$+$};
\end{tikzpicture}
\begin{tikzpicture}[scale=0.3,thick]
\draw (4,0)--(4,1.5)--(6,2.5)--(6,4);
\draw (4,4)--(4,2.5)--(6,1.5)--(6,0);
\draw (7,2) node{$+$};
\end{tikzpicture}
\begin{tikzpicture}[scale=0.3,thick]
\draw (8,4)--(8,1.5)--(8.75,1)--(8.75,0);
\draw (8.75,1)--(9.5,1.5)--(9.5,2.5)--(10.25,3)--(10.25,4);
\draw (10.25,3)--(11,2.5)--(11,0);
\draw (11.75,2) node{$+$};
\end{tikzpicture}
\begin{tikzpicture}[scale=0.3,thick]
\draw (21.5,0)--(21.5,1.25)--(22.5,2)--(21.5,2.75)--(21.5,4);
\draw (21.5,1.25)..controls (20.5,2) and (20,2.5) .. (20,3)--(20,4);
\draw (21.5,2.75)..controls (20.5,2) and (20,1.5).. (20,1)--(20,0);
\draw (23.5,2) node{$+$};
\end{tikzpicture}
\begin{tikzpicture}[scale=0.3,thick]
\draw (3,4)--(3,1.5)--(2.25,1)--(2.25,0);
\draw (2.25,1)--(1.5,1.5)--(1.5,2.5)--(0.75,3)--(0.75,4);
\draw (0.75,3)--(0,2.5)--(0,0);
\draw (4,2) node{$+$};
\end{tikzpicture}
\begin{tikzpicture}[scale=0.3,thick]
\draw (17,0)--(17,1.25)--(16,2)--(17,2.75)--(17,4);
\draw (17,1.25)..controls (18,2) and (18.5,2.5) .. (18.5,3)--(18.5,4);
\draw (17,2.75)..controls (18,2) and (18.5,1.5).. (18.5,1)--(18.5,0);
\draw (19.5,2) node{$+$};
\end{tikzpicture}
\begin{tikzpicture}[scale=0.3,thick]
\draw (1,0)--(1,1.25)--(0,2)--(1,2.75)--(1,4);
\draw (1,2.75)--(3,1.25)--(3,0);
\draw (1,1.25)--(3,2.75)--(3,4);
\draw (3,1.25)--(4,2)--(3,2.75);
\end{tikzpicture}

\item \cite{nui} $\A=$As, $\C=$As with the n.u.i. mixed distributive law:

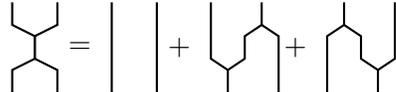
\begin{tikzpicture}[scale=0.3,thick]
\draw (0,0)--(0,1)--(1,1.5)--(1,2.5)--(0,3)--(0,4);
\draw (2,0)--(2,1)--(1,1.5)--(1,2.5)--(2,3)--(2,4);
\draw (3,2) node{$=$};
    \end{tikzpicture}
\begin{tikzpicture}[scale=0.3,thick]
\draw (4,0)--(4,4);
\draw (6,0)--(6,4);
\draw (7,2) node{$+$};
\end{tikzpicture}
\begin{tikzpicture}[scale=0.3,thick]
\draw (8,4)--(8,1.5)--(8.75,1)--(8.75,0);
\draw (8.75,1)--(9.5,1.5)--(9.5,2.5)--(10.25,3)--(10.25,4);
\draw (10.25,3)--(11,2.5)--(11,0);
\draw (11.75,2) node{$+$};
\end{tikzpicture}
\begin{tikzpicture}[scale=0.3,thick]
\draw (3,4)--(3,1.5)--(2.25,1)--(2.25,0);
\draw (2.25,1)--(1.5,1.5)--(1.5,2.5)--(0.75,3)--(0.75,4);
\draw (0.75,3)--(0,2.5)--(0,0);
\end{tikzpicture}

\item \cite{AsZinb} $\A=$As, $\C=$Zinb with the semi-Hopf mixed distributive law, denoted by $\prec$ the generating operation of Zinbiel: 

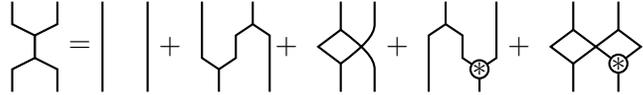
\begin{tikzpicture}[scale=0.3,thick]
\draw (0,0)--(0,1)--(1,1.5)--(1,2.5)--(0,3)--(0,4);
\draw (2,0)--(2,1)--(1,1.5)--(1,2.5)--(2,3)--(2,4);
\draw (3,2) node{$=$};
\draw (4,0)--(4,4);
\draw (6,0)--(6,4);
\draw (7,2) node{$+$};
\end{tikzpicture}
\begin{tikzpicture}[scale=0.3,thick]
\draw (8,4)--(8,1.5)--(8.75,1)--(8.75,0);
\draw (8.75,1)--(9.5,1.5)--(9.5,2.5)--(10.25,3)--(10.25,4);
\draw (10.25,3)--(11,2.5)--(11,0);
\draw (11.75,2) node{$+$};
\end{tikzpicture}
\begin{tikzpicture}[scale=0.3,thick]
\draw (17,0)--(17,1.25)--(16,2)--(17,2.75)--(17,4);
\draw (17,1.25)..controls (18,2) and (18.5,2.5) .. (18.5,3)--(18.5,4);
\draw (17,2.75)..controls (18,2) and (18.5,1.5).. (18.5,1)--(18.5,0);
\draw (19.5,2) node{$+$};
\end{tikzpicture}
\begin{tikzpicture}[scale=0.3,thick]
\draw (3,4)--(3,1.5)--(2.25,1)--(2.25,0);
\draw (2.25,1)--(1.5,1.5)--(1.5,2.5)--(0.75,3)--(0.75,4);
\draw (0.75,3)--(0,2.5)--(0,0);
\draw[fill=white] (2.25,1) circle(0.4);
\draw(2.25,1) node {$*$};
\draw (4,2) node{$+$};
\end{tikzpicture}
\begin{tikzpicture}[scale=0.3,thick]
\draw (1,0)--(1,1.25)--(0,2)--(1,2.75)--(1,4);
\draw (1,2.75)--(3,1.25)--(3,0);
\draw (1,1.25)--(3,2.75)--(3,4);
\draw (3,1.25)--(4,2)--(3,2.75);
\draw[fill=white] (3,1.25) circle(0.4);
\draw(3,1.25) node {$*$};
\end{tikzpicture},

with $*=\prec \circ (id + (1 \  2))$.

\item {\cite{MagInf}}  $\A=$Mag, $\C=$Mag with the magmatic mixed distributive law:

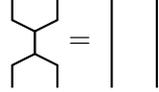
\begin{tikzpicture}[scale=0.3,thick]
\draw (0,0)--(0,1)--(1,1.5)--(1,2.5)--(0,3)--(0,4);
\draw (2,0)--(2,1)--(1,1.5)--(1,2.5)--(2,3)--(2,4);
\draw (3,2) node{$=$};
    \end{tikzpicture}
\begin{tikzpicture}[scale=0.3,thick]
\draw (4,0)--(4,4);
\draw (6,0)--(6,4);
\end{tikzpicture}

\item {\cite{MagInf}}  $\A=$Mag$\infty$, $\C=$Mag$\infty$ with the infinite magmatic mixed distributive law:

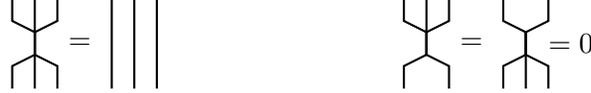
\begin{tikzpicture}[scale=0.3,thick]
\draw (0,0)--(0,1)--(1,1.5)--(1,2.5)--(0,3)--(0,4);
\draw (2,0)--(2,1)--(1,1.5)--(1,2.5)--(2,3)--(2,4);
\draw (1,0)--(1,4);
\draw (3,2) node{$=$};
    \end{tikzpicture}
\begin{tikzpicture}[scale=0.3,thick]
\draw (4,0)--(4,4);
\draw (5,0)--(5,4);
\draw (6,0)--(6,4);
\end{tikzpicture}
\hspace{3cm}
\begin{tikzpicture}[scale=0.3,thick]
\draw (0,0)--(0,1)--(1,1.5)--(1,2.5)--(0,3)--(0,4);
\draw (2,0)--(2,1)--(1,1.5)--(1,2.5)--(2,3)--(2,4);
\draw (1,1.5)--(1,4);
\draw (3,2) node{$=$};
    \end{tikzpicture}
    \begin{tikzpicture}[scale=0.3,thick]
\draw (0,0)--(0,1)--(1,1.5)--(1,2.5)--(0,3)--(0,4);
\draw (2,0)--(2,1)--(1,1.5)--(1,2.5)--(2,3)--(2,4);
\draw (1,0)--(1,2.5);
\draw (3,2) node{$= 0$};
    \end{tikzpicture}

\item {\cite{NAPPL}} $\A=$NAP, $\C=$PreLie with the Livernet mixed distributive law:

\begin{tikzpicture}[scale=0.3,thick]
\draw (0,0)--(0,1)--(1,1.5)--(1,2.5)--(0,3)--(0,4);
\draw (2,0)--(2,1)--(1,1.5)--(1,2.5)--(2,3)--(2,4);
\draw (3,2) node{$=$};
    \end{tikzpicture}
\begin{tikzpicture}[scale=0.3,thick]
\draw (4,0)--(4,4);
\draw (6,0)--(6,4);
\draw (7,2) node{$+$};
\end{tikzpicture}
\begin{tikzpicture}[scale=0.3,thick]
\draw (3,4)--(3,1.5)--(2.25,1)--(2.25,0);
\draw (2.25,1)--(1.5,1.5)--(1.5,2.5)--(0.75,3)--(0.75,4);
\draw (0.75,3)--(0,2.5)--(0,0);
\draw (4,2) node{$+$};
\end{tikzpicture}
\begin{tikzpicture}[scale=0.3,thick]
\draw (17,0)--(17,1.25)--(16,2)--(17,2.75)--(17,4);
\draw (17,1.25)..controls (18,2) and (18.5,2.5) .. (18.5,3)--(18.5,4);
\draw (17,2.75)..controls (18,2) and (18.5,1.5).. (18.5,1)--(18.5,0);
\end{tikzpicture}

\item \cite{GBO} $\A=$Nil, $\C=$Nil with nil mixed distributive law:

\begin{tikzpicture}[scale=0.3,thick]
\draw (0,0)--(0,1)--(1,1.5)--(1,2.5)--(0,3)--(0,4);
\draw (2,0)--(2,1)--(1,1.5)--(1,2.5)--(2,3)--(2,4);
\draw (3,2) node{$=$};
    \end{tikzpicture} 
\begin{tikzpicture}[scale=0.3,thick]
\draw (4,0)--(4,4);
\draw (6,0)--(6,4);
\draw (7,2) node{$-$};
\end{tikzpicture}
\begin{tikzpicture}[scale=0.3,thick]
\draw (0.75,0)--(0.75,1.5)--(0,2)--(0.75,2.5)--(0.75,4);
\draw (0.75,1.5)--(1.5,2)--(0.75,2.5);
\draw (2.25,0)--(2.25,4);
\draw (3.25,2) node{$-$};
\end{tikzpicture}
\begin{tikzpicture}[scale=0.3,thick]
\draw (0.75,0)--(0.75,1.5)--(0,2)--(0.75,2.5)--(0.75,4);
\draw (0.75,1.5)--(1.5,2)--(0.75,2.5);
\draw (-0.75,0)--(-0.75,4);
\draw (2.5,2) node{$+$};
\end{tikzpicture}
\begin{tikzpicture}[scale=0.3,thick]
\draw (0.75,0)--(0.75,1.5)--(0,2)--(0.75,2.5)--(0.75,4);
\draw (0.75,1.5)--(1.5,2)--(0.75,2.5);
\draw (2.75,0)--(2.75,1.5)--(2,2)--(2.75,2.5)--(2.75,4);
\draw (2.75,1.5)--(3.5,2)--(2.75,2.5);
\end{tikzpicture}
\item \cite{GBO} $\A=$ Dup, $\C=$Dup with the following mixed distributive law:

\begin{tikzpicture}[scale=0.3,thick]
\draw (0,0)--(0,1)--(1,1.5)--(1,3.5)--(0,4)--(0,5);
\draw (2,0)--(2,1)--(1,1.5)--(1,3.5)--(2,4)--(2,5);
\draw[fill=white] (1,1.5) circle(0.6);
\draw[fill=white] (1,3.5) circle(0.6);
\draw (1,1.5) node {$\succ$};
\draw (1,3.5) node {$\succ$};
\draw (3,2) node{$=$};
\draw (4,0)--(4,5);
\draw (6,0)--(6,5);
\draw (7,2) node{$+$};
\end{tikzpicture}
\begin{tikzpicture}[scale=0.3,thick]
\draw (8,4.5)--(8,1.5)--(8.75,1)--(8.75,-0.5);
\draw (8.75,1)--(9.5,1.5)--(9.5,2.5)--(10.25,3)--(10.25,4.5);
\draw (10.25,3)--(11,2.5)--(11,-0.5);
\draw[fill=white] (8.75,1) circle(0.6);
\draw[fill=white] (10.25,3) circle(0.6);
\draw (8.75,1) node {$\succ$};
\draw (10.25,3) node {$\succ$};
\draw (11.75,2.5) node{$+$};
\end{tikzpicture}
\begin{tikzpicture}[scale=0.3,thick]
\draw (3,4.5)--(3,1.5)--(2.25,1)--(2.25,-0.5);
\draw (2.25,1)--(1.5,1.5)--(1.5,2.5)--(0.75,3)--(0.75,4.5);
\draw (0.75,3)--(0,2.5)--(0,-0.5);
\draw[fill=white] (2.25,1) circle(0.6);
\draw[fill=white] (0.75,3) circle(0.6);
\draw (2.25,1) node {$\succ$};
\draw (0.75,3) node {$\succ$};
\end{tikzpicture}

\begin{tikzpicture}[scale=0.3,thick]
\draw (0,0)--(0,1)--(1,1.5)--(1,3.5)--(0,4)--(0,5);
\draw (2,0)--(2,1)--(1,1.5)--(1,3.5)--(2,4)--(2,5);
\draw[fill=white] (1,1.5) circle(0.6);
\draw[fill=white] (1,3.5) circle(0.6);
\draw (1,1.5) node {$\prec$};
\draw (1,3.5) node {$\succ$};
\draw (3,2) node{$=$};
\end{tikzpicture}
\begin{tikzpicture}[scale=0.3,thick]
\draw (8,4.5)--(8,1.5)--(8.75,1)--(8.75,-0.5);
\draw (8.75,1)--(9.5,1.5)--(9.5,2.5)--(10.25,3)--(10.25,4.5);
\draw (10.25,3)--(11,2.5)--(11,-0.5);
\draw[fill=white] (8.75,1) circle(0.6);
\draw[fill=white] (10.25,3) circle(0.6);
\draw (8.75,1) node {$\succ$};
\draw (10.25,3) node {$\prec$};
\end{tikzpicture}
\hspace{2cm}
\begin{tikzpicture}[scale=0.3,thick]
\draw (0,0)--(0,1)--(1,1.5)--(1,3.5)--(0,4)--(0,5);
\draw (2,0)--(2,1)--(1,1.5)--(1,3.5)--(2,4)--(2,5);
\draw[fill=white] (1,1.5) circle(0.6);
\draw[fill=white] (1,3.5) circle(0.6);
\draw (1,1.5) node {$\succ$};
\draw (1,3.5) node {$\prec$};
\draw (3,2) node{$=$};
\end{tikzpicture}
\begin{tikzpicture}[scale=0.3,thick]
\draw (3,4.5)--(3,1.5)--(2.25,1)--(2.25,-0.5);
\draw (2.25,1)--(1.5,1.5)--(1.5,2.5)--(0.75,3)--(0.75,4.5);
\draw (0.75,3)--(0,2.5)--(0,-0.5);
\draw[fill=white] (2.25,1) circle(0.6);
\draw[fill=white] (0.75,3) circle(0.6);
\draw (2.25,1) node {$\succ$};
\draw (0.75,3) node {$\prec$};
\end{tikzpicture}

\begin{tikzpicture}[scale=0.3,thick]
\draw (0,0)--(0,1)--(1,1.5)--(1,3.5)--(0,4)--(0,5);
\draw (2,0)--(2,1)--(1,1.5)--(1,3.5)--(2,4)--(2,5);
\draw[fill=white] (1,1.5) circle(0.6);
\draw[fill=white] (1,3.5) circle(0.6);
\draw (1,1.5) node {$\prec$};
\draw (1,3.5) node {$\prec$};
\draw (3,2) node{$=$};
\draw (4,0)--(4,5);
\draw (6,0)--(6,5);
\draw (7,2) node{$+$};
\end{tikzpicture}
\begin{tikzpicture}[scale=0.3,thick]
\draw (8,4.5)--(8,1.5)--(8.75,1)--(8.75,-0.5);
\draw (8.75,1)--(9.5,1.5)--(9.5,2.5)--(10.25,3)--(10.25,4.5);
\draw (10.25,3)--(11,2.5)--(11,-0.5);
\draw[fill=white] (8.75,1) circle(0.6);
\draw[fill=white] (10.25,3) circle(0.6);
\draw (8.75,1) node {$\prec$};
\draw (10.25,3) node {$\prec$};
\draw (11.75,2.5) node{$+$};
\end{tikzpicture}
\begin{tikzpicture}[scale=0.3,thick]
\draw (3,4.5)--(3,1.5)--(2.25,1)--(2.25,-0.5);
\draw (2.25,1)--(1.5,1.5)--(1.5,2.5)--(0.75,3)--(0.75,4.5);
\draw (0.75,3)--(0,2.5)--(0,-0.5);
\draw[fill=white] (2.25,1) circle(0.6);
\draw[fill=white] (0.75,3) circle(0.6);
\draw (2.25,1) node {$\prec$};
\draw (0.75,3) node {$\prec$};
\end{tikzpicture}

\end{itemize}
Note that if $(\C^c-_{\alpha}\A)$-bialgebras satisfy the rigidity theorem with a given confluence law $(\alpha)$, so does $(\A^c-_{\bar{\alpha}}\C)$-bialgebras with the confluence law $\bar{\alpha}$ obtained as the dual of $\alpha$, or, graphically, as the horizontal mirror image of $\alpha$.
\end{example}

\begin{remark}
Some cases are not obtained by duality of the same product, for instance in the dendriform case, the isomorphism given by Foissy's mixed distributive law (\cite{Foissy})  in arity two (between $\operatorname{Dend}(2)$ and $\operatorname{Dend}(2)^*$) is given by: 
\begin{align*}
 \varphi_2: & 1\succ2 \mapsto (1\succ2)^* \\
  & 1 \prec 2 \mapsto (2 \prec 1)^*.
\end{align*}
The matrix of $\varphi_2$ is diagonalisable but admits $-1$ as an eigenvalue. Indeed, Foissy uses two different dendriform products defined on dendriform algebras and computes the mixed distributive laws between one and the dual of the other.

In the 2-as case computed by Loday et Ronco in \cite{nui}, the isomorphism given by Hopf and n.u.i. mixed distributive laws in arity two  is diagonalisable but admits $-1$ and $3$ as an eigenvalue. It would be interesting to know if there exists two different 2-as products such that when looking at the mixed distributive laws between one and the dual of the other, one recovers Loday and Ronco's mixed distributive laws. 
\end{remark}

\subsection{PreLie Case.} \label{PLcase}

We consider here the PreLie operad and the rooted tree basis introduced by Chapoton and Livernet in \cite{ChapLiv}: the free PreLie algebra on a vector space $V$ is spanned by rooted (non planar) trees with vertices indexed by $V$.  This example was the motivation for the introduction of confluence laws.

We recall that the relation satisfied by a PreLie product $\curvearrowleft$ is given by: 
\begin{equation*}
(x \curvearrowleft y) \curvearrowleft z - x \curvearrowleft (y \curvearrowleft z)=
(x \curvearrowleft z) \curvearrowleft y - x \curvearrowleft (z \curvearrowleft y)
\end{equation*}

Combinatorially, the product $T \curvearrowleft S$ is the sum over all possible ways to add an edge between a vertex of $T$ and the root of $S$. The root of the obtained tree is the root of $T$.
 
The dual coproduct is then given by the sum over all possible ways to delete an edge in the tree:
\begin{equation*}
\Delta(T)=\sum_{a \in E(T)} R_a(T)  \otimes L_a(T),
\end{equation*}
where $R_a(T)$ is the connected component of $T-\{a\}$ containing the root of $T$ and $L_a(T)$ is the other connected component.

\begin{remark}
This coproduct is obtained by taking only connected components in Connes-Kreimer coproduct.
\end{remark}

To apply the rigidity theorem to some algebras, we compute the associated confluence law:

\begin{proposition} \label{RelCompPL}
The PreLie product and its dual coproduct satisfy the following confluence law, for $T \in \operatorname{PreLie}(n)$ and $S \in \operatorname{PreLie}(k)$, $p_1, \ldots, p_{n+k}$ some primitive elements in a given $\operatorname{PreLie}^c-\operatorname{PreLie}$-bialgebra:
\begin{align*} 
\Delta & \left(T \curvearrowleft  S  \right) (p_1, \ldots, p_{n+k}) = n \times T(p_1, \ldots, p_n) \otimes S(p_{n+1}, \ldots, p_{n+k}) \\
& + (T \curvearrowleft S_1) \otimes S_2 (p_1, \ldots, p_{n+k}) + (T_1 \curvearrowleft S) \otimes T_2 (p_1, \ldots, p_{n+k}) \\
&+ T_1 \otimes (T_2 \curvearrowleft S) (p_1, \ldots, p_{n+k}),
\end{align*}
where $\Delta(T) = T_1 \otimes T_2$ and $\Delta(S) = S_1 \otimes S_2$.
\begin{figure}[h!]
\begin{tikzpicture}[scale=0.3,thick]
\draw (0,0)--(0,1)--(1,1.5)--(1,2.5)--(0,3)--(0,4);
\draw (2,0)--(2,1)--(1,1.5)--(1,2.5)--(2,3)--(2,4);
\draw (3,2) node{$=$};
\draw (4,5) node{$\#$};
\draw (4,0)--(4,4);
\draw (6,0)--(6,4);
\draw (7,2) node{$+$};
\end{tikzpicture}
\begin{tikzpicture}[scale=0.3,thick]
\draw (8,4)--(8,1.5)--(8.75,1)--(8.75,0);
\draw (8.75,1)--(9.5,1.5)--(9.5,2.5)--(10.25,3)--(10.25,4);
\draw (10.25,3)--(11,2.5)--(11,0);
\draw (11.75,2) node{$+$};
\end{tikzpicture}
\begin{tikzpicture}[scale=0.3,thick]
\draw (3,4)--(3,1.5)--(2.25,1)--(2.25,0);
\draw (2.25,1)--(1.5,1.5)--(1.5,2.5)--(0.75,3)--(0.75,4);
\draw (0.75,3)--(0,2.5)--(0,0);
\draw (4,2) node{$+$};
\end{tikzpicture}
\begin{tikzpicture}[scale=0.3,thick]
\draw (17,0)--(17,1.25)--(16,2)--(17,2.75)--(17,4);
\draw (17,1.25)..controls (18,2) and (18.5,2.5) .. (18.5,3)--(18.5,4);
\draw (17,2.75)..controls (18,2) and (18.5,1.5).. (18.5,1)--(18.5,0);
\end{tikzpicture}
\end{figure}

The previous definition of the confluence law is easier to apply. We also state another definition, closer to the formal definition of confluence law but strictly equivalent to the previous one, for $T \in \operatorname{PreLie}(n)$ and $S \in \operatorname{PreLie}(k)$:

\begin{equation*}
\Delta_S(T(p_1, \ldots, p_n))=\sum_{\substack{S_1, \ldots, S_k \\ S_i \in \operatorname{PreLie}(l_i)\\ l_1 + \ldots + l_k = n}} S_1(p_{\sigma(1)}, \ldots, p_{\sigma(l_1)}) \otimes \ldots \otimes S_k(p_{\sigma(n-l_k+1)}, \ldots, p_{\sigma(n)}),
\end{equation*}
where the sum is taken over products $S_1, \ldots, S_k$ satisfying $\mu(S,S_1, \ldots, S_k) (p_1, \ldots, p_n) \ni T(p_1, \ldots, p_n)$, where $\mu$ is the operad composition of PreLie and only one element S is taken in the orbit by the action of the symmetric group (for instance, only one representative of $(p_1 \curvearrowleft p_2) \curvearrowleft p_3 - p_1 (\curvearrowleft p_2 \curvearrowleft p_3 )= (p_1 \curvearrowleft p_3) \curvearrowleft p_2 - p_1 (\curvearrowleft p_3 \curvearrowleft p_2 ) $ is chosen.)

\end{proposition}

\begin{proof}
For the first expression, the decomposition is done following the deleted edge: it can be between $S$ and $T$, in $S$, or in $T$. If it is in $T$, it can be between the root of $T$ and the root of $S$ or somewhere else.

The second expression follows quite easily from the definition of the dual coproduct. 
\end{proof}

It has been pointed to the authors that the sketch of this confluence law can also be found in  \cite{LMB}.

\begin{example}
Let us give explicitely the second formulation of confluence laws on some example. For $T= \begin{tikzpicture}[scale=0.5]
\draw(1,1)--(0,0)--(-2,2);
\draw(-1,1)--(0,2);
\draw[fill=white](0,0) circle(0.4);
\draw[fill=white](1,1) circle(0.4);
\draw[fill=white](-1,1) circle(0.4);
\draw[fill=white](0,2) circle(0.4);
\draw[fill=white](-2,2) circle(0.4);
\draw(0,0) node{4};
\draw(1,1) node{5};
\draw(-1,1) node{2};
\draw(0,2) node{3};
\draw(-2,2) node{1};
\end{tikzpicture}$ and $S=\begin{tikzpicture}[scale=0.5]
\draw(1,1)--(0,0)--(-1,1);
\draw[fill=white](0,0) circle(0.4);
\draw[fill=white](1,1) circle(0.4);
\draw[fill=white](-1,1) circle(0.4);
\draw(0,0) node{2};
\draw(1,1) node{3};
\draw(-1,1) node{1};
\end{tikzpicture}$, we get: 
\begin{align*}
\Delta_S(T)&(p_1, \ldots, p_5) = \As[p_3] \otimes \Av[p_4][p_2][p_5] \otimes \As[p_1] + \As[p_1] \otimes \Av[p_4][p_2][p_5] \otimes \As[p_3]   \\
& + \As[p_5] \otimes \Al[p_4][p_2][p_3] \otimes \As[p_1] +\As[p_1] \otimes \Al[p_4][p_2][p_3] \otimes \As[p_5] + \As[p_5] \otimes \Al[p_4][p_2][p_1] \otimes \As[p_3]   \\
& + \As[p_3] \otimes \Al[p_4][p_2][p_1] \otimes \As[p_5] + \As[p_5] \otimes \As[p_4] \otimes \Av[p_2][p_1][p_3] + \Av[p_2][p_1][p_3] \otimes \As[p_4] \otimes \As[p_5].
\end{align*}
Note that there are two terms per way to choose one edge in a left subtree and one edge in the right subtree of a node due to symmetries of $S$.

\end{example}

\begin{proposition}
Applying the algorithm, the idempotent is given by:
\begin{equation*}
e=\operatorname{id} + \sum_{n \geq 2}\sum_{\substack{T \in HT_n \\ \operatorname{Linext}(S) \cap \operatorname{Linext}(T)\neq\emptyset}} \frac{(-1)^{n-1}}{|\operatorname{Aut}(T)|!} S \circ T^*,
\end{equation*}
where we see a rooted tree $T$ as the Hasse diagram of a poset $P(T)$ with a unique minimal element (the root) and $\operatorname{Linext(T)}$ is the set of linear extension of the poset $P(T)$.
\end{proposition}

\begin{proof}
We show that this idempotent vanishes on any non trivial rooted tree by an inclusion-exclusion principle. For any $S \in HT_n$, the coefficient in front of a tree $T \in HT_n$ in $e(S)$, with $T$ having $k$ edges different from the ones in $S$, is obtained from coproducts deleting these $k$ edges and $p-k$ others  by:
\begin{equation*}
c \times \sum_{p \geq k} \binom{n-1-k}{p-k} (-1)^p=0.
\end{equation*}

Hence the result.
\end{proof}

We now apply rigidity theorem for PreLie algebras to three examples in the literature.

\begin{example}[Box trees] \label{Box trees}
We use the obtained criterion to give a new proof of the freeness of the algebra of partitioned trees introduced in \cite{FPL}. Partitioned trees are equivalent to box trees introduced in \cite{mar2}.

Let us consider a quadruple $\left(L,V,R,E\right)$, where
\begin{itemize}
\item $L$ is a finite set called the set of \emph{labels},
\item $V$ is a partition of $L$ called the set of \emph{vertices},
\item $R$ is an element of $V$ called the root,
\item $E$ is a map from $V-\{R\}$ to $L$ called the set of \emph{edges}. 
\end{itemize} 
We will denote by $\tilde{E}$, the map from $V-\{R\}$ to $V$ which associates to a vertex $v$ the vertex $v'$ containing the label $E\left(v\right)$. The pair $\left(V,\tilde{E}\right)$ is then an oriented graph, with vertices labelled by subsets of $L$.

\begin{definition}[\cite{FPL},\cite{mar2}]
A quadruple $\left(L,V,R,E\right)$ is a \emph{box tree} if and only if the graph $\left(V,\tilde{E}\right)$ is a tree, rooted in $R$, with edges oriented toward the root.

A label $l$ is called \emph{parent} of a vertex $v$ if $E\left(v\right)=l$.
\end{definition}	

In Figure \ref{exemple box tree}, an example of box trees is presented. The root is the double rectangle.

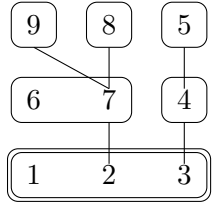
\begin{figure} 
\centering 
\begin{tikzpicture}[scale=1]
\draw[black, rounded corners] (-0.35,-0.35) rectangle(2.35,0.35);
\draw[black, rounded corners] (-0.3,-0.3) rectangle(2.3,0.3);
\draw[black, rounded corners] (-0.3,0.7) rectangle(1.3,1.3);
\draw[black, rounded corners] (1.7,0.7) rectangle(2.3,1.3);
\draw[black, rounded corners] (-0.3,1.7) rectangle(0.3,2.3);
\draw[black, rounded corners] (0.7,1.7) rectangle(1.3,2.3);
\draw[black, rounded corners] (1.7,1.7) rectangle(2.3,2.3);
\draw (1,0.7) -- (1,0.15);
\draw (2,0.7) -- (2,0.15);
\draw (1,1.7) -- (1,1.15);
\draw (0,1.7) -- (1,1.15);
\draw (2,1.7) -- (2,1.15);
\draw (0,0) node{$1$};
\draw (1,0) node{$2$};
\draw (2,0) node{$3$};
\draw (2,1) node{$4$};
\draw (2,2) node{$5$};
\draw (0,1) node{$6$};
\draw (1,1) node{$7$};
\draw (1,2) node{$8$};
\draw (0,2) node{$9$};
\end{tikzpicture}
\caption{\label{exemple box tree} A box tree. }
\end{figure}

The product is given by the natural PreLie product on trees.

 On this algebra, we define the following coproduct:
\begin{equation*}
\Delta(x)=\sum_{\substack{e \in E(x),\\ e: a \rightarrow b}} \frac{1}{|a|} R(x-\{e\}) \otimes L(x-\{e\}).
\end{equation*}
This coproduct satisfies the previous confluence law. Hence the associated algebra is PreLie free, with primitive given by trees with no edges.
\end{example}

\begin{example}[Hypertrees] \label{hypertree}
 The bijection between decorated hypertrees and a pair given by some type of box trees and decorated sets motivated the introduction of the following product on hypertrees introduced by Berge in \cite{Berge} and studied by the second author in \cite{mar1}, \cite{mar2} and \cite{mar3}:

\begin{definition}[\cite{Berge}]
		 A \emph{hypergraph (on a set $V$)} is an ordered pair $(V,E)$ where $V$ is a finite set and $E$ is a collection of elements of cardinality at least two, belonging to the power set $\mathcal{P}(V)$. The elements of $V$ are called \emph{vertices} and those of $E$ are called \emph{edges}.
		 \end{definition}
		 
		 An example of hypergraph is presented in figure~\ref{hypfig}.

		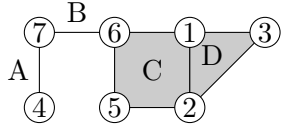
\begin{figure} 
		\centering
\begin{tikzpicture}[scale=1]
\draw (0,0) -- (0,1) node[midway, left]{A};
\draw (0,1) -- (1,1) node[midway, above]{B};
\draw[black, fill=gray!40] (1,1) -- (1,0) -- (2,0) -- (2,1) -- (1,1);
\draw(1.5,0.5) node{C};
\draw[black, fill=gray!40] (2,0) -- (2,1)--(3,1) -- (2,0);
\draw(2.30,0.70) node{D};
\draw[black, fill=white] (0,0) circle (0.2);
\draw[black, fill=white] (0,1) circle (0.2);
\draw[black, fill=white] (1,1) circle (0.2);
\draw[black, fill=white] (1,0) circle (0.2);
\draw[black, fill=white] (2,0) circle (0.2);
\draw[black, fill=white] (2,1) circle (0.2);
\draw[black, fill=white] (3,1) circle (0.2);
\draw(0,0) node{$4$};
\draw(0,1) node{$7$};
\draw(1,1) node{$6$};
\draw(1,0) node{$5$};
\draw(2,1) node{$1$};
\draw(2,0) node{$2$};
\draw(3,1) node{$3$};
\end{tikzpicture}
\caption{An example of hypergraph on $\{1,2,3,4,5,6,7\}$.\label{hypfig}}
	
		\end{figure}

		\begin{definition} Let $H=(V,E)$ be a hypergraph.
		
		 A \emph{walk from a vertex or an edge $d$ to a vertex or an edge $f$ in $H$} is an alternating sequence of vertices and edges beginning by $d$ and ending by $f$ $(d, \ldots, e_i, v_i, e_{i+1}, \ldots, f)$ where for all $i$, $v_i \in V$, $e_i \in E$ and $\{v_i,v_{i+1}\} \subseteq e_i$. The \textit{length} of a walk is the number of edges and vertices in the walk.
		\end{definition}

\begin{example} In the previous example, there are several walks from $4$ to $2$: $(4,A,7,B,6,C,2)$ and $(4,A,7,B,6,C,1,D,3,D,2)$. A walk from $C$ to $3$ is $(C,1,D,3)$.
\end{example}

\begin{definition} A \emph{hypertree} is a non-empty hypergraph $H$ such that, given any distinct vertices $v$ and $w$ in $H$, 
\begin{itemize}
\item there exists a walk from $v$ to $w$ in $H$ with distinct edges $e_i$, i.e. $H$ is \emph{connected},
\item and this walk is unique, i.e. $H$ has \emph{no cycles}. 
 \end{itemize}

The pair $H=(V,E)$ is called \emph{hypertree on $V$}. If $V$ is the set $ \{ 1, \ldots, n \}$, then $H$ is called an \emph{hypertree on $n$ vertices}.
\end{definition}

An example of an hypertree is presented in figure \ref{htfig}.

\begin{figure}

\begin{center}
\begin{tikzpicture}[scale=1]
\draw[black, fill=gray!40] (0,0) -- (1,1) -- (0,1) -- (0,0);
\draw(1,0) -- (1,1);
\draw[black, fill=white] (0,0) circle (0.2);
\draw[black, fill=white] (0,1) circle (0.2);
\draw[black, fill=white] (1,1) circle (0.2);
\draw[black, fill=white] (1,0) circle (0.2);
\draw(0,0) node{$4$};
\draw(0,1) node{$1$};
\draw(1,1) node{$2$};
\draw(1,0) node{$3$};
\end{tikzpicture}
\caption{An example of hypertree on $\{1,2,3,4\}$.\label{htfig} }
\end{center}

\end{figure}
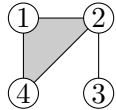

Let us consider a rooted hypertree $H$, i.e. a hypertree with a distinguished vertex. Given an edge $e$ of $H$, there is one vertex of $e$ which is the nearest from the root of $H$ in $e$: let us call it the \emph{petiole} of $e$. 

We define a PreLie product on rooted hypertrees $H \curvearrowleft G$ as the sum of all the ways to graft the root of $G$ on a vertex $v$ of $H$, where the grafting is given by adding an edge between $v$ and $r$.

\begin{example} We represent below the product of two rooted hypertrees:
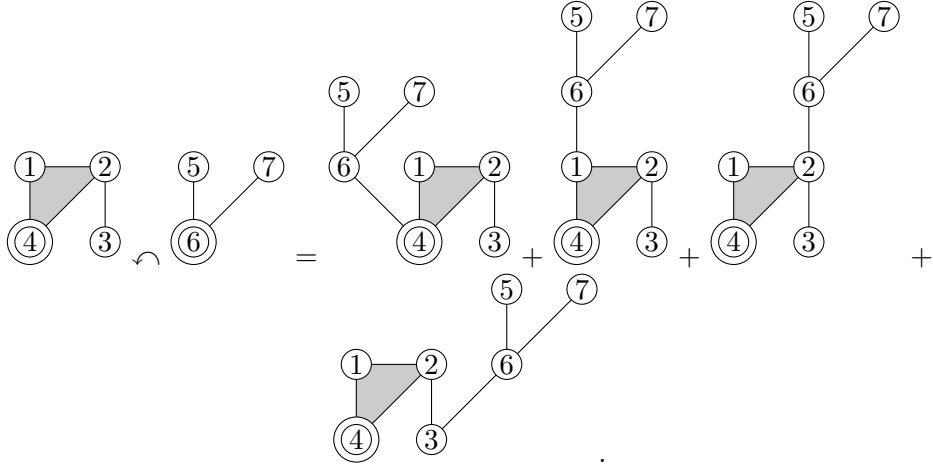
\begin{center}
\begin{tikzpicture}[scale=1]
\draw[black, fill=gray!40] (0,0) -- (1,1) -- (0,1) -- (0,0);
\draw(1,0) -- (1,1);
\draw[black, fill=white] (0,0) circle (0.3);
\draw[black, fill=white] (0,1) circle (0.2);
\draw[black, fill=white] (1,1) circle (0.2);
\draw[black, fill=white] (1,0) circle (0.2);
\draw[black, fill=white] (0,0) circle (0.2);
\draw(0,0) node{$4$};
\draw(0,1) node{$1$};
\draw(1,1) node{$2$};
\draw(1,0) node{$3$};
\end{tikzpicture}
$\curvearrowleft$
\begin{tikzpicture}[scale=1]
\draw[black] (0,1)--(0,0) -- (1,1) ;
\draw[black, fill=white] (0,0) circle (0.3);
\draw[black, fill=white] (0,0) circle (0.2);
\draw[black, fill=white] (1,1) circle (0.2);
\draw[black, fill=white] (0,1) circle (0.2);
\draw(0,0) node{$6$};
\draw(1,1) node{$7$};
\draw(0,1) node{$5$};
\end{tikzpicture}
$=$
\begin{tikzpicture}[scale=1]
\draw (0,0)--(-1,1);

\draw[black, fill=gray!40] (0,0) -- (1,1) -- (0,1) -- (0,0);
\draw(1,0) -- (1,1);
\draw[black, fill=white] (0,0) circle (0.3);
\draw[black, fill=white] (0,1) circle (0.2);
\draw[black, fill=white] (1,1) circle (0.2);
\draw[black, fill=white] (1,0) circle (0.2);
\draw[black, fill=white] (0,0) circle (0.2);
\draw(0,0) node{$4$};
\draw(0,1) node{$1$};
\draw(1,1) node{$2$};
\draw(1,0) node{$3$};

\draw[black] (-1,2)--(-1,1) -- (0,2) ;
\draw[black, fill=white] (-1,1) circle (0.2);
\draw[black, fill=white] (0,2) circle (0.2);
\draw[black, fill=white] (-1,2) circle (0.2);
\draw(-1,1) node{$6$};
\draw(0,2) node{$7$};
\draw(-1,2) node{$5$};
\end{tikzpicture}
$+$
\begin{tikzpicture}[scale=1]
\draw (0,1)--(0,2);

\draw[black, fill=gray!40] (0,0) -- (1,1) -- (0,1) -- (0,0);
\draw(1,0) -- (1,1);
\draw[black, fill=white] (0,0) circle (0.3);
\draw[black, fill=white] (0,1) circle (0.2);
\draw[black, fill=white] (1,1) circle (0.2);
\draw[black, fill=white] (1,0) circle (0.2);
\draw[black, fill=white] (0,0) circle (0.2);
\draw(0,0) node{$4$};
\draw(0,1) node{$1$};
\draw(1,1) node{$2$};
\draw(1,0) node{$3$};

\draw[black] (0,3)--(0,2) -- (1,3) ;
\draw[black, fill=white] (0,2) circle (0.2);
\draw[black, fill=white] (1,3) circle (0.2);
\draw[black, fill=white] (0,3) circle (0.2);
\draw(0,2) node{$6$};
\draw(1,3) node{$7$};
\draw(0,3) node{$5$};
\end{tikzpicture}
$+$
\begin{tikzpicture}[scale=1]
\draw (1,1)--(1,2);

\draw[black, fill=gray!40] (0,0) -- (1,1) -- (0,1) -- (0,0);
\draw(1,0) -- (1,1);
\draw[black, fill=white] (0,0) circle (0.3);
\draw[black, fill=white] (0,1) circle (0.2);
\draw[black, fill=white] (1,1) circle (0.2);
\draw[black, fill=white] (1,0) circle (0.2);
\draw[black, fill=white] (0,0) circle (0.2);
\draw(0,0) node{$4$};
\draw(0,1) node{$1$};
\draw(1,1) node{$2$};
\draw(1,0) node{$3$};

\draw[black] (1,3)--(1,2) -- (2,3) ;
\draw[black, fill=white] (1,2) circle (0.2);
\draw[black, fill=white] (2,3) circle (0.2);
\draw[black, fill=white] (1,3) circle (0.2);
\draw(1,2) node{$6$};
\draw(2,3) node{$7$};
\draw(1,3) node{$5$};
\end{tikzpicture}
$+$
\begin{tikzpicture}[scale=1]
\draw (1,0)--(2,1);

\draw[black, fill=gray!40] (0,0) -- (1,1) -- (0,1) -- (0,0);
\draw(1,0) -- (1,1);
\draw[black, fill=white] (0,0) circle (0.3);
\draw[black, fill=white] (0,1) circle (0.2);
\draw[black, fill=white] (1,1) circle (0.2);
\draw[black, fill=white] (1,0) circle (0.2);
\draw[black, fill=white] (0,0) circle (0.2);
\draw(0,0) node{$4$};
\draw(0,1) node{$1$};
\draw(1,1) node{$2$};
\draw(1,0) node{$3$};

\draw[black] (2,2)--(2,1) -- (3,2) ;
\draw[black, fill=white] (2,1) circle (0.2);
\draw[black, fill=white] (3,2) circle (0.2);
\draw[black, fill=white] (2,2) circle (0.2);
\draw(2,1) node{$6$};
\draw(3,2) node{$7$};
\draw(2,2) node{$5$};
\end{tikzpicture}.

\end{center}
\end{example}

On this algebra, we define the following coproduct:
\begin{equation*}
\Delta(T)=\sum_{{e \in E(T)}} R(T-\{e\}) \otimes L(T-\{e\}).
\end{equation*}
This coproduct satisfies the previous confluence law. Hence the associated algebra is PreLie free, with primitive given by hypertrees with no binary edge, i.e. edge of cardinality two.
\end{example}

\begin{example}[Fat trees] \label{Fat trees}
In \cite{mar2} and \cite{PhD}, the second author also studied the notion of (rooted) fat trees introduced in \cite{Zas}:

\begin{definition}[ \cite{Zas}] 
A \emph{fat tree on a set $V$} is a partition of $V$, whose parts are called \emph{vertices}, together with edges linking elements of different vertices, such that: 
\begin{itemize}
\item a \emph{walk} on the fat tree is an alternating sequence $\left(a_1, b_1, c_1, a_2, \ldots, c_n\right)$, where for every $i$, $a_i$ and $c_i$ are elements of different vertices and $b_i$ is an edge between $a_i$ and $c_i$, and for every $i$ between $1$ and $n-1$, $c_i$ and $a_{i+1}$ are elements of the same vertex;
\item For every pair of elements of different vertices $\left(a,c\right)$, there exists one and only one walk from $a$ to $c$.
\end{itemize}

A \emph{rooted fat tree} is a fat tree with a distinguished element called the \emph{root}.
\end{definition}

In Figure \ref{exemple fat}, an example of a rooted fat tree is presented. The root is circled.

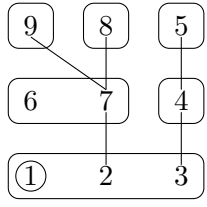
\begin{figure} 
\centering
\begin{tikzpicture}[scale=1]
\draw[black, rounded corners] (-0.3,-0.3) rectangle(2.3,0.3);
\draw[black, rounded corners] (-0.3,0.7) rectangle(1.3,1.3);
\draw[black, rounded corners] (1.7,0.7) rectangle(2.3,1.3);
\draw[black, rounded corners] (-0.3,1.7) rectangle(0.3,2.3);
\draw[black, rounded corners] (0.7,1.7) rectangle(1.3,2.3);
\draw[black, rounded corners] (1.7,1.7) rectangle(2.3,2.3);
\draw (1,0.85) -- (1,0.15);
\draw (2,0.85) -- (2,0.15);
\draw (1,1.85) -- (1,1.15);
\draw (0,1.85) -- (1,1.15);
\draw (2,1.85) -- (2,1.15);
\draw[black, fill=white] (0,0) circle (0.2);
\draw (0,0) node{$1$};
\draw (1,0) node{$2$};
\draw (2,0) node{$3$};
\draw (2,1) node{$4$};
\draw (2,2) node{$5$};
\draw (0,1) node{$6$};
\draw (1,1) node{$7$};
\draw (1,2) node{$8$};
\draw (0,2) node{$9$};
\end{tikzpicture}
\caption{ \label{exemple fat} A rooted fat tree.}
\end{figure}

The vector space of rooted fat trees can easily be endowed with a PreLie product given by the tree structure.

 On this algebra, we define the following coproduct:
\begin{equation*}
\Delta(x)=\sum_{\substack{e \in E(x),\\ e: a\in A \rightarrow b \in B}} \frac{1}{|A||B|} R(x-\{e\}) \otimes L(x-\{e\}).
\end{equation*}
This coproduct satisfies the previous confluence law. Hence the associated algebra is PreLie free, with primitive given by fat trees with no edges.

\end{example}

\begin{remark}
This example shows that the formula \eqref{crochet} should contain a sum over permutations of primitive elements to take into account the action of the symmetric group, and automorphism groups of operations, which does not appear in \cite{GBO}. For instance, in PreLie(3), we have, denoting the generating product and coproduct respectively by $\Delta$ and $\mu$:
\begin{equation*}
(\Delta \otimes id) \circ \Delta \circ \mu \circ \left( \mu \otimes id - id \otimes \mu \right) (p_1, p_2, p_3) = p_1 \otimes p_2 \otimes p_3 + p_1 \otimes p_3 \otimes p_2,
\end{equation*}
for any triples of primitive elements $p_i$. The second term of the sum will always appear if the first one appear because of the equality:
\begin{equation*}
\mu \circ \left( \mu \otimes id - id \otimes \mu \right) (p_1, p_2, p_3) = \mu \circ \left( \mu \otimes id - id \otimes \mu \right) (p_1, p_3, p_2).
\end{equation*}
\end{remark}

\subsection{Perm Case.}

We now apply the rigidity theorem in dual case to Operad Perm, introduced by Chapoton in \cite{Perm2} endowed with the usual vector space basis given by pointed sets. 

We recall that the relation satisfied by a Perm product $\times$ is given by: 
\begin{align*}
(x \times y) \times z &= x \times (y \times z) \\
&= x \times (z \times y)
\end{align*}

Combinatorially, the product $S_1 \times S_2$ is the pointed set obtained by the union of $S_1$ and $S_2$ pointed in the pointed element of $S_1$.

The dual coproduct is then given by the sum over all possible ways to split the set $S$ in two pointed set $S_1$ and $S_2$ such that $S$ is the union of $S_1$ and $S_2$ and the pointed element of $S_1$ is the pointed element of $S$:
\begin{equation*}
\Delta(S)=\sum_{\substack{S_1 \cup S_2 = S \\ p(S) = p(S_1)}} S_1  \otimes S_2,
\end{equation*}
where, for any pointed set $T$, $p(T)$ denotes the pointed element of $T$.

In this case, the idempotent is given explicitly by the following formula:
\begin{equation*}
e= \sum_{n \geq 1} \frac{(-1)^{n-1}}{n!} n^{n-1} D_n,
\end{equation*} 
where $D_n$ is defined recursively by:
\begin{align*} 
D_{1} &=id  \\
D_{n+1} &= \times \circ (D_n \otimes id) \circ \Delta.
\end{align*}

\begin{remark}
We use the notation $D$ because the diagram of $D_2 = \times \circ \Delta$ is a diamond.
\end{remark}

\begin{proof}
We show by induction that $D_k(\{\mathbf{1}, \ldots, n\}) = \prod_{i=1}^k (n-i) (k+1)^{n-k-1}$. The iterated coproduct satisfies the following relation:
\begin{equation*}
\Delta_k (\{\mathbf{1}, \ldots, n\}) = \sum_{p=k}^{n-1} (n-p) \binom{n-1}{p-1} \Delta_{k-1}((\{\mathbf{1}, \ldots, p\})).
\end{equation*}
Indeed, only terms in the coproduct whose left part has at least $k$ elements remain after $k$ coproducts. To choose such a decomposition, we have $(n-p) \binom{n-1}{p-1}$ choices.
\end{proof}

To apply the rigidity theorem to some algebras, we compute the associated confluence law:

\begin{proposition}
The Perm product and its dual coproduct satisfy on primitive elements the following confluence law, with $\overline{T}=\sum_{k \geq 1} \frac{1}{k} \delta_{T \in \operatorname{Perm}(k)} $ and $\dot{T}=T + \delta_{ (\Delta(T))_2 \in \operatorname{Perm}(1)} (\Delta(T))_2 \times (\Delta(T))_1$, for any element $T$, where $\delta$ is the Kronecker symbol:
\begin{align*}\label{RelCompPerm}
\Delta(T \times S) =& T \otimes \dot{S} + T_1 \otimes (T_2 \times S + \dot{S} \times \overline{T_2}) + (T_1 \times S) \otimes T_2 + (T \times S_1) \otimes S_2 \\ 
+& (T \times \overline{S_2}) \otimes \dot{S_1}+ (T_1 \times S_1) \otimes (T_2 \times \overline{S_2} + S_2 \times \overline{T_2}) \\
+& (T_1 \times \overline{S_2}) \otimes (T_2 \times S_1 +  S_1 \times \overline{T_2}),
\end{align*}
where $\Delta(T) = T_1 \otimes T_2$ and $\Delta(S) = S_1 \otimes S_2$.
\end{proposition}

\begin{proof}
The coproduct of the product is obtained by merging $T$ and $S$, forgetting the pointed element of $S$ and then splitting the pointed set into two pointed sets, the left one containing the pointed element of $T$. The equality is obtained by considering the different cases: 
\begin{itemize}
\item if the splitting separates $S$ and $T$, $T$ is on the left side and $S$ can be pointed in any element, 
\item  if the splitting splits $T$ only, the right part of the coproduct can be pointed in an element of $t$ or any element of $S$, 
\item  if the splitting splits $S$ only, the pointed element of the left part of the coproduct is fixed and the right part can contain the pointed element of $S$ (which is in $S_1$) or not and then be pointed in any element.
\item if the splitting splits both $S$ and $T$, the pointed element of $S$ can be on the right or on the left part and the pointed element of the right part can be chosen in $T$ or in $S$.
\end{itemize}
\end{proof}

\subsection{NAP Case.}

We consider here the NAP operad (see \cite{ChapLiv} and \cite{dzuma}). Livernet has proven in \cite{NAPPL} the existence of a mixed distributive law such that any conilpotent PreLie coNAP bialgebra satisfying the associated mixed distributive law is free and cofree over its primitive elements. We study here NAP coNAP bialgebras.

 The NAP product $\cdot$ on a vector space $V$ satisfies the following relation, for all $x$,$y$ and $z$ in $V$:
\begin{equation*} 
(x \cdot y) \cdot z=(x \cdot z) \cdot y.
\label{NAP rel}
\end{equation*}

A basis of a free NAP algebra over a vector space is given by the set of rooted trees on $n$ vertices labelled by elements from $V$, denoted by $RT_n(V)$. 
The product of two trees $T$ and $S$ in this algebra is then the tree $T \cdot S$ obtained by grafting the root of $S$ to the root of $T$ (see Figure \ref{T.NAPS}).

\begin{figure}
\begin{tikzpicture}
[level distance=10mm,
every node/.style={circle,inner sep=1pt, draw}]
\node  {1} [grow'=up]
child {node {2}}
child {node {3}};
\end{tikzpicture}
$\cdot$
\begin{tikzpicture}
[level distance=10mm,
every node/.style={circle,inner sep=1pt, draw}]
\node  {4} [grow'=up]
child {node {5}};
\end{tikzpicture}
$=$
\begin{tikzpicture}
[level distance=10mm,
every node/.style={circle,inner sep=1pt, draw}]
\node  {1} [grow'=up]
child {node {2}}
child {node {3}}
child {node {4}
child {node {5}}};
\end{tikzpicture}
\caption{NAP product \label{T.NAPS}}
\end{figure}
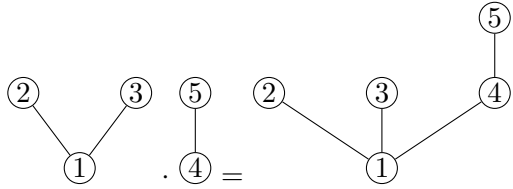

The coproduct $\Delta(T)$ of a tree $T$ in this algebra, given as the dual of the product, is the sum of all possible trees obtained by deleting an edge attached to the root of $T$. The part containing the root of $T$ is then the left part of the coproduct (see Figure \ref{DNAP(T)}). 

\begin{figure}
$\Delta:
\begin{tikzpicture}
[level distance=10mm,
every node/.style={circle,inner sep=1pt, draw}]
   \node  {1} [grow'=up]
child {node {2}}
child {node {3}
child{node{4}}};
\end{tikzpicture} 
\mapsto
\begin{tikzpicture}
[level distance=10mm,
every node/.style={circle,inner sep=1pt, draw}]
\node  {1} [grow'=up]
child {node {2}};
\end{tikzpicture} \otimes
\begin{tikzpicture}
[level distance=10mm,
every node/.style={circle,inner sep=1pt, draw}]
\node  {3} [grow'=up]
child{node{4}} ;
\end{tikzpicture}+ 
\begin{tikzpicture}
[level distance=10mm,
every node/.style={circle,inner sep=1pt, draw}]
\node  {1} [grow'=up]
child {node {3}
child{node{4}}};
\end{tikzpicture} \otimes
2$
\caption{NAP coproduct \label{DNAP(T)}}
\end{figure}
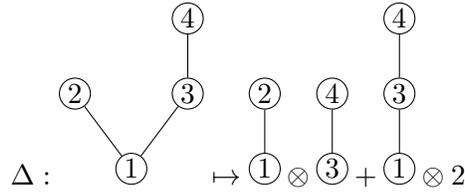

To apply the rigidity theorem to some algebras, we compute the associated confluence law (which is a mixed distributive law):
\begin{proposition}
The NAP product $\cdot$ and its associated dual coproduct satisfy the following mixed distributive law:
\begin{equation*} \label{comprelNAP}
\Delta (T \cdot S) = T \otimes S + T_1 \cdot S \otimes T_2, 
\end{equation*}
where $\Delta(T) = T_1 \otimes T_2$ (Sweedler's notation).
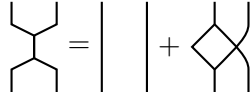
\begin{figure}[h!]
\begin{tikzpicture}[scale=0.3, thick]
\draw (0,0)--(0,1)--(1,1.5)--(1,2.5)--(0,3)--(0,4);
\draw (2,0)--(2,1)--(1,1.5)--(1,2.5)--(2,3)--(2,4);
\draw (3,2) node {$=$};
\draw (4,0)--(4,4);
\draw (6,0)--(6,4);
\draw (7,2) node{$+$};
\draw (9,0)--(9,1)--(8,2)--(9,3)--(9,4);
\draw (9,1)..controls (10,2) and (10.5,2.5) .. (10.5,3)--(10.5,4);
\draw (9,3)..controls (10,2) and (10.5,1.5).. (10.5,1)--(10.5,0);
\end{tikzpicture}
\end{figure}
\end{proposition}

\begin{proof}
$\Delta (T \cdot S)$ is the sum of all possibilities of deleting an edge containing the root of $T$ in the tree obtained by adding an edge $e$ between the root of $T$ and the one of $S$. The result comes from the following decomposition: either the edge deleted is $e$ or it is an edge of $T$.
\end{proof}

The previous reasoning gives us the following expression for idempotent, checked in the proof below:
\begin{equation*}
e= \sum_{n \geq 1} (-1)^{n-1} \frac{n}{n!} D_n,
\end{equation*} 
where $D_n$ is defined recursively by:
\begin{align*} 
D_{1}=id  \\
D_{n+1} = \cdot \circ (D_n \otimes id) \circ \Delta.
\end{align*}

\begin{proof}
 The operator $D_n$ can be rewritten in terms of rooted trees by \scalebox{0.5}{\begin{tikzpicture}[scale=0.5]
[every node/.style={circle,draw}]
\node  {1} [grow'=up]
child {node {2}}
child {node {3}}
child {node {...}}
child {node {n}};
\end{tikzpicture} } $\circ \scalebox{0.5}{\begin{tikzpicture}[scale=0.5]
[every node/.style={circle,draw}]
\node  {1} [grow'=up]
child {node {2}}
child {node {3}}
child {node {...}}
child {node {n}};
\end{tikzpicture} }^*$.
The partial sum $e_k= \sum_{n \geq 1}^k (-1)^{n-1} \frac{n}{n!} D_n$ vanishes over all rooted trees on at least two vertices, whose root has at most $k$ children.

\end{proof}

Any free PreLie algebra can be endowed with a natural structure of free NAP algebra. Note that it is not the case for non free PreLie algebras: for instance, the PreLie structure on Mag operad cannot be endowed with a NAP product (see \cite{EBBDODM}). 

The examples developed in the subsection \ref{PLcase} can then be seen as NAP-algebras when considering grafting to the root only.

\begin{example}[Box trees] 
Considering Example \ref{Box trees} endowed with the natural NAP structure associated to the PreLie product, we can define the following coNAP coproduct:
\begin{equation*}
\Delta(x)=\sum_{\substack{e \in E_r(x),\\ e: r \rightarrow b}} \frac{1}{|r|} R(x-\{e\}) \otimes L(x-\{e\}),
\end{equation*}
where $E_r(x)$ is the set of edges of $x$ adjacent to the root $r$ of $x$.

This coproduct satisfies the previous mixed distributive law. Hence the associated algebra is NAP free, with primitive given by trees with no edges.
\end{example}

\begin{example}[Hypertrees] 
Considering Example \ref{hypertree}, we define a NAP product on rooted hypertrees $H \cdot G$ as the hypertree obtained by grafting the root $r_G$ of $G$ on the root $r_H$ of $H$, where the grafting is given by adding an edge between $r_G$ and $r_H$.

On this algebra, we define the following coproduct:
\begin{equation*}
\Delta(T)=\sum_{{e \in E_r(T)}} R(T-\{e\}) \otimes L(T-\{e\}),
\end{equation*}
where $E_r(T)$ is the set of edges of $T$ adjacent to the root $r$ of $x$.

This coproduct satisfies the previous mixed distributive law. Hence the associated algebra is NAP free, with primitive given by hypertrees with no binary edge, i.e. edge of cardinality two.
\end{example}

\begin{example}[Fat trees] 
Considering Example \ref{Fat trees} endowed with the natural NAP structure associated to the PreLie product, we can define the following coNAP coproduct:
\begin{equation*}
\Delta(x)=\sum_{\substack{e \in E_r(x),\\ e: r\in R \rightarrow b \in B}} \frac{1}{|R||B|} R(x-\{e\}) \otimes L(x-\{e\}),
\end{equation*}
where $E_r(x)$ is the set of edges of $x$ adjacent to the root $r$ of $x$.

This coproduct satisfies the previous mixed distributive law. Hence the associated algebra is NAP free, with primitive given by trees with no edges.

\end{example}

\subsection{PAN Case.}

We consider here the Koszul dual of the NAP operad, denoted by PAN. The PAN product $\leftharpoondown$ on a vector space $V$ satisfies the following relation, for all $x$,$y$ and $z$ in $V$:
\begin{align*} 
x \leftharpoondown (y \leftharpoondown z) = 0 \\
(x \leftharpoondown y) \leftharpoondown z = (x \leftharpoondown z) \leftharpoondown y
\label{PAN rel}
\end{align*}

\begin{proposition}
The operad PAN is the Koszul dual of the operad NAP. 
\end{proposition}

\begin{proof} 
We use the methods of V. Dotsenko and E. Hoffbeck (see \cite{Hoffbeck}) to determine the Koszul dual of NAP operad. Let us recall the relation in NAP:
\begin{equation*} 
(x \cdot y) \cdot z=(x \cdot z) \cdot y.
\end{equation*}
The orthogonal of this relation is precisely the relations of PAN.
\end{proof}

A basis of a free PAN algebra over a vector space $V$ is given by the set of pointed sets on $n$ elements labelled by elements from $V$. 
Denoting in bold the pointed element of the set, the product of two pointed sets $\{\mathbf{x_1}, \ldots, x_k\}$ and $\{\mathbf{y_1}, \ldots, y_l\}$ in this algebra is then the pointed set $\{\mathbf{x_1}, \ldots, x_k, y_1, \ldots, y_l\}$ if $l=1$, $0$ otherwise.

The coproduct $\Delta(u)$ of a pointed set $u$ pointed in $x$ in this algebra, given as the dual of the product, is given by:
\begin{equation} 
\Delta(u) = \sum_{\substack{\mathrm{a} \in u \\ \mathrm{a} \neq x}} u-\{\mathrm{a}\} \otimes \{\text{\textbf{a}} \},
\end{equation}
where the set $u-\{\mathrm{a}\}$ is pointed in $x$.

To apply the rigidity theorem to some algebras, we compute the associated confluence law:

\begin{proposition}
The PAN product $\leftharpoondown$ and its associated dual coproduct satisfy the following confluence law on the primitive elements of any conilpotent PAN-bialgebra:
\begin{equation} \label{comprelPAN}
\Delta (u \leftharpoondown v) = \delta_{v \in \operatorname{PAN}(1)} \left(u \otimes v + u_1 \cdot v \otimes u_2 \right), 
\end{equation}
where $\Delta(u) = u_1 \otimes u_2$ (Sweedler's notation) and $\delta$ is the Kronecker symbol.
\begin{figure}[h!]
\begin{tikzpicture}[scale=0.3, thick]
\draw (-1,0)--(-1,1)--(-2,1.5)--(-2,2.5)--(-1,3)--(-1,4);
\draw (-3,0)--(-3,1)--(-2,1.5)--(-2,2.5)--(-3,3)--(-3,4);
\draw (1.5,2) node {$= \delta_{v \in \operatorname{PAN}(1)}$};
\draw (5,0)--(5,4);
\draw (7,0)--(7,4);
\draw (8,2) node{$+$};
\draw (10,0)--(10,1)--(9,2)--(10,3)--(10,4);
\draw (10,1)..controls (11,2) and (11.5,2.5) .. (11.5,3)--(11.5,4);
\draw (10,3)..controls (11,2) and (11.5,1.5).. (11.5,1)--(11.5,0);
\end{tikzpicture}
\end{figure}
\end{proposition}

\begin{proof} If $v \notin \mathcal{F}_1$, $\Delta (u \leftharpoondown v)=0$. Otherwise, the coproduct is obtained by either separating $u$ and $v$ or splitting $u$.
\end{proof}

The previous reasoning gives us the following expression for idempotent, checked in the proof below:
\begin{equation*}
e= \sum_{n \geq 1} (-1)^{n-1} \frac{n}{n!} D_n,
\end{equation*} 
where $D_n$ is defined recursively by:
\begin{align*} \label{defDn}
D_{1}=id  \\
D_{n+1} = \cdot \circ (D_n \otimes id) \circ \Delta.
\end{align*}

\begin{proof}
 We show by induction that $D_k(\{\mathbf{1}, \ldots, n\}) = \prod_{i=1}^{k-1} (n-i) $. Indeed, the iterated coproduct satisfies the following relation:
\begin{equation}
\Delta_k (\{\mathbf{1}, \ldots, n\}) = (n-1) \Delta_{k-1}((\{\mathbf{1}, \ldots, n-1\})).
\end{equation}
The equality then comes from Newton binomial theorem.
\end{proof}

\begin{center}
\emph{coPAN-Perm bialgebras}
\end{center}

Any Perm algebra can be endowed with a structure of coPAN bialgebras by considering the coproduct dual to the natural structure of PAN-algebra.

The following confluence law follows directly from the definitions:
\begin{proposition}
The Perm product $\times$ and the associated dual coPAN-coproduct satisfy the following confluence law on the primitive elements of any conilpotent coPAN-Perm-bialgebra:
\begin{equation} \label{comprelPermPAN}
\Delta (u \times v) = \delta_{v \in \operatorname{PAN}(1)} u \otimes v + u_1 \times v \otimes u_2 + u \times v_1 \otimes v_2 + \delta_{v_1 \in \operatorname{PAN}(1)} u \times v_2 \otimes v_1, 
\end{equation}
where $\Delta(u) = u_1 \otimes u_2$ and $\Delta(v) = v_1 \otimes v_2$ (Sweedler's notation).
\end{proposition}

\subsection{Associative, Leibniz, Poisson and Zinbiel case.}

We now consider rigidity theorems obtained from operads which underlying free algebras is the tensor algebra (with different products): Associative, Leibniz, Poisson and Zinbiel. These operads were introduced respectively in \cite{Leibniz}, \cite{Poisson} and \cite{Zinbiel}.

The relations satisfied by these operads are respectively:
\begin{itemize}
\item for the Associative product $\cdot$, $(x\cdot y) \cdot z = x\cdot (y \cdot z)$ (the $\cdot$ will be sometimes omitted)
\item for the Leibniz product $[.,.]$, $[[x,y],z] = [x,[y,z]] + [[x,z],y]$,
\item for the Poisson products $\times$ and $\lbrace . , . \rbrace$, $\times$ is (associative) commutative, $\lbrace \ , \ \rbrace$ is a Lie bracket and $\lbrace x \times y, z \rbrace = x \times \lbrace y, z \rbrace + \lbrace x, z \rbrace \times y$
\item for the Zinbiel product $\prec$, $(x \prec y) \prec z = x \prec (y \prec z) + x \prec (z \prec y)$
\end{itemize} 

Using relations, Poisson can be interpreted as $\operatorname{Comm} \circ \operatorname{Lie}$ (commutative products of Lie brackets of elements). A basis of Poisson operad is then given by the usual Lyndon basis of Lie algebras, with commutative terms sorted by non increasing order. The obtained terms are then naturally bracketed, the Lie brackets being determined by left-to-right minima and being naturally bracketed as Lyndon words (for instance, $\mathbf{4652371}$ stands for $\lbrace \lbrace 4,6 \rbrace ,5\rbrace \times \lbrace 2, \lbrace 3,7 \rbrace \rbrace \times 1$). We denote these representation in bold to distinguish it from the representation below.

To identify it as associative elements, one has to use the injection of Lie algebras into associative algebras defined by $i: [a,b] \mapsto i(a) \cdot i(b) - i(b) \cdot i(a)$ and the injection of commutative algebras into associative algebras defined by $\iota: a \times b \mapsto \operatorname{sh}(\iota(a), \iota(b) )$.

Using the above injections and with implicit $\cdot$ between elements in the right part of the equality, we obtain a basis of the Poisson operad. The elements of small arities are then given by:
\begin{itemize}
\item in arity $2$, $\mathbf{12}=1 \times 2 = 12+21$ and $\mathbf{21}=[1,2] = 12-21$
\item in arity $3$, 
\begin{align*}
\mathbf{123}&= \lbrace 1 , \lbrace 2,3 \rbrace \rbrace = 123-132-231+321\\
\mathbf{132}&= \lbrace \lbrace1, 3 \rbrace, 2\rbrace = 132-312-213+231\\
\mathbf{213}&= 2 \times \lbrace1,3 \rbrace = 213-231+132-312\\
\mathbf{231}&=\lbrace 2,3 \rbrace \times 1= 231-321+123-132\\
\mathbf{312} &=3 \times \lbrace 1, 2 \rbrace= 312-321+123-213\\
\mathbf{321} &=3 \times 2 \times 1= 123+132+213+231+312+321 
\end{align*}
\end{itemize}

Combinatorially, these products correspond on the tensor algebra to:
\begin{itemize}
\item the concatenation for the associative product: 
\begin{equation*}
(x_1 \ldots x_p) \cdot (x_{p+1} \ldots x_{p+q}) = x_1 \ldots x_{p+q}
\end{equation*}

\item for the Leibniz product $[.,.]$, 
\begin{equation*}
[x_1 \ldots x_p, y_1, \ldots y_q] = \sum_{I \sqcup J =\{2, \ldots q\}} (-1)^{|J|} x_1 \ldots x_p y_{j_{|J|}} \ldots y_{j_1} y_1 y_{i_1} \ldots y_{i_{|I|}}
\end{equation*}
with $i_1<i_2< \ldots <i_{|I|}$ and $j_1<j_2<\ldots<j_{|J|}$,

\item the shuffle product for the Poisson product $\times$ and
 the commutator of concatenation for the Poisson bracket $\lbrace . , . \rbrace$: 
 
 \begin{equation*}
 \lbrace (x_1 \ldots x_p) , (x_{p+1} \ldots x_{p+q}) \rbrace= x_1 \ldots x_px_{p+1} \ldots x_{p+q} - x_{p+1} \ldots x_{p+q}x_1 \ldots x_p
 \end{equation*}
 and
  \begin{equation*}
(x_1 \ldots x_p) \times (x_{p+1} \ldots x_{p+q}) = \operatorname{sh}_{p,q}(x_1 \ldots x_p, x_{p+1} \ldots x_{p+q}). 
 \end{equation*}

\item the halfshuffle for the Zinbiel product: 
\begin{equation*}
(x_1 \ldots x_p) \prec (x_{p+1} \ldots x_{p+q}) = x_1 \operatorname{sh}_{p-1,q}(x_2 \ldots x_p, x_{p+1} \ldots x_{p+q}),
\end{equation*}
where $sh_{\alpha, \beta}(x_1, \ldots, x_{\alpha+\beta}) = \sum_{\pi} x_{\pi(1)} \ldots x_{\pi(\alpha+\beta)}$, with $\pi^{-1}(1)< \ldots < \pi^{-1}(\alpha)$ and $\pi^{-1}(\alpha+1)< \ldots <\pi^{-1}(\alpha+\beta)$.

\end{itemize}

The associated dual coproducts are then given on the tensor algebra by:
\begin{itemize}
\item the deconcatenation for the associative product: 
\begin{equation*}
\Delta_{\cdot}(x_1 \ldots x_p) = \sum_{i=1}^{p-1} x_1 \ldots x_i \otimes x_{i+1} \ldots x_{p},
\end{equation*}

\item for the Leibniz coproduct $[.,.]$, 
\begin{equation*}
\Delta_{[.,.]}(x_1 \ldots x_p)= \sum_{i=1}^{p-1} \sum_{j > i} (-1)^{j-(i+1)} x_1 \ldots x_i \otimes x_j \operatorname{sh} (x_{j+1} \ldots x_p, x_{j-1} \ldots x_{i+1}),
\end{equation*}

\item the coshuffle coproduct for the Poisson product $\times$ and
 the commutator of deconcatenation for the Poisson bracket $\lbrace . , . \rbrace$: 
 
 \begin{equation*}
\Delta_{\lbrace .,. \rbrace} (x_1 \ldots x_p) = \sum_{i=1}^{p-1} x_1 \ldots x_i \otimes x_{i+1} \ldots x_{p} - x_{i+1} \ldots x_{p} \otimes x_1 \ldots x_i,
 \end{equation*}
 and
  \begin{equation*}
\Delta_{\times}(x_1 \ldots x_p) = \sum_{\substack{I \sqcup J =\{1, \ldots p\} \\ I,J \neq \emptyset}} x_{i_1} \ldots x_{i_{|I|}} \otimes x_{j_1}  \ldots x_{j_{|J|}}
 \end{equation*}
with $i_1<i_2< \ldots <i_{|I|}$ and $j_1<j_2<\ldots<j_{|J|}$.

\item the cohalfshuffle for the Zinbiel coproduct: 
\begin{equation*}
\Delta_{\prec} (x_1 \ldots x_p) =\sum_{I \sqcup J =\{2, \ldots p\}} x_1 x_{i_1} \ldots x_{i_{|I|}} \otimes x_{j_1}  \ldots x_{j_{|J|}},
\end{equation*}
with $i_1<i_2< \ldots <i_{|I|}$ and $j_1<j_2<\ldots<j_{|J|}$.
\end{itemize}

The mixed distributive laws for the different relations are given by:
\begin{itemize}
\item \emph{Associative-Associative case:} n.u.i. mixed distributive law  proven in \cite{nui} (written on Example \ref{2.1.5}),
\item \emph{Associative-Zinbiel case:} semi-Hopf mixed distributive law proven in\cite{AsZinb} (written on Example \ref{2.1.5}). Note that the product $\ast=\prec \circ (id+(12))$ used is exactly the shuffle product $\times$,

\item \emph{Associative-Leibniz case:} To express a confluence law, we need some operators expressed in terms of coassociative coproduct and Leibniz product:
\begin{equation*}
\operatorname{RV}(x) = \delta_{x \in \F_1} + \delta_{x_1 \in \F_1} [\operatorname{RV}(x_2), x_1],
\end{equation*}
where $\Delta_{\cdot}(x)=x_1 \otimes x_2$ and $\delta$ is the Kronecker symbol,

\begin{equation*}
x \cdot y = \delta_{y \in \F_1} [x,y] + \delta_{y_2 \in \F_1} [x.y_1, y_2],
\end{equation*}

\begin{equation*}
[\epsilon,x_1 \ldots x_n] = x_1 \ldots x_n + \sum_{I \sqcup J=\{1, \ldots, n\}} \operatorname{RV}(x_{i_1} \ldots x_{i_{|I|}}) \cdot x_{j_1} \ldots x_{j_{J}},
\end{equation*}
where $I=\{i_1, \ldots, i_{|I|}\}$, $i_1 < \ldots <i_{|I|}$, $J=\{j_1, \ldots, j_{|J|}\}$, $j_1 < \ldots, j_{|J|}$ and the sequences are extracted thanks to deconcatenation and concatenation defined just above.

The confluence law is then given combinatorially, according to where the deconcatenation occurs, by:
\begin{equation*}
\Delta_{\cdot}([u,v]) = u \otimes [\epsilon,v] + \left(\Delta_{\cdot}(u)\right)_1  \otimes \left[ \left(\Delta_{\cdot}(u)\right)_2 , v \right] + u \cdot \left(\Delta_{\cdot}([\epsilon,v])\right)_1   \otimes \left(\Delta_{\cdot}([\epsilon,v])\right)_2
\end{equation*}

\item \emph{Associative-Poisson case:} 
The mixed distributive law is given combinatorially by:                                                                                                                                                                                                                                                                                                                                                                                           
\begin{align*}
\Delta_{\cdot}( u \times v)&= u \otimes v + v \otimes u + (u \times v_1) \otimes v_2 + v_1 \otimes (u \times v_2) + (u_1 \times v) \otimes u_2  \\
&+ u_1 \otimes (u_2 \times v) + (u_1 \times v_1) \otimes (u_2 \times v_2)\text{\ (Hopf)}
\end{align*}
 and
\begin{equation*}
 \Delta_{\cdot}(\lbrace u, v\rbrace) = u \otimes v -v \otimes u +u_1 \otimes u_2 \cdot v + u \cdot v_1 \otimes v_2 -v_1 \otimes v_2 \cdot u - v \cdot u_1 \otimes u_2,
\end{equation*}
where $\Delta_{\cdot}(u)=u_1 \otimes u_2$ and $\Delta_{\cdot}(v) = v_1 \otimes v_2$. 

\item \emph{Zinbiel-Leibniz case:} 
Remark first that the concatenation can be obtained recursively by:
\begin{equation*}
u \cdot v = \sum_{k \geq 1} \delta_{v_1, \ldots, v_k \in \F_1} [[ \ldots [[u, v_1],v_2]\ldots],v_k]
\end{equation*}
where $(\operatorname{id} \otimes \ldots \otimes \Delta_{\prec}) \circ \ldots \circ \Delta_{\prec}(v)=v_1 \otimes \ldots \otimes v_k$.

We can then define as previously the operations $\operatorname{RV}$ and $[\epsilon,v]$.
 
The confluence law is thus given combinatorially by:
\begin{equation*}
\Delta_{\prec}([u,v]) = u \otimes [\epsilon,v] + \left(\Delta_{\prec}(u)\right)_1  \otimes \left[ \left(\Delta_{\prec}(u)\right)_2 , v \right] + \left[ \left(\Delta_{\prec}(u)\right)_1 , v \right] \otimes \left(\Delta_{\prec}(u)\right)_2
\end{equation*}

No term can be obtained by splitting $v$ because a term $uv_{l_1} \ldots v_{l_k} \otimes v_{r_1} \ldots v_{r_p}$ can be obtained thanks to $k+1$ different elements of $[u,v]$ (according to how the elements are mixed) and then the coefficient of this term is exactly $\sum_{p=0}^k (-1)^p \binom{k}{p} = 0$.

\end{itemize}

\subsection{Dendriform and Tridendriform case.}

These cases are treated in the article \cite{DendTridend}. The relations obtained for the bidendriform bialgebra are different from the ones obtained by Foissy in \cite{Foissy}.

\subsection{2-as and dipt case.}

We now compute mixed distributive laws for Operads 2-as and Dipt, introduced by J.-L. Loday and M. Ronco in \cite{nui} and \cite{dipt}.

We recall that the relation satisfies by 2-as products $\ast$ and $\cdot$ are given by: 
\begin{align*}
(x \ast y) \ast z &= x \ast (y \ast z) \\
(x \cdot y) \cdot z&= x \cdot (y \cdot z)
\end{align*}

Combinatorially, the free 2-associative algebra on a vector space $V$ is spanned by words on planar trees with leaves decorated by elements of $V$. The product $ \ast$ is then the concatenation of trees in the word and $\cdot$ is a grafting on a new root.

Following the dual case for the associative operad, the products and their dual coproducts are linked by the mixed distributive law given in the following array:
\begin{center}
\begin{tabular}{|c|c|c|}
\hline
& $\ast$ & $\cdot$ \\\hline
$\Delta_\ast$ & n.u.i. & 0 \\ \hline
$\Delta_{\cdot}$ & 0 & n.u.i \\ \hline
\end{tabular}
\end{center}

We recall that the relation satisfies by dipterous products $\star$ and $\prec$ are given by: 
\begin{align*}
(x \star y) \star z &= x \star (y \star z) \\
(x \prec y) \prec z&= x \prec (y \star z)
\end{align*}

Combinatorially (see \cite{nui} and \cite{dipt}), the free dipterous algebra on a vector space $V$ is spanned by words on planar trees with leaves decorated by elements of $V$. The product $ \star$ is then the concatenation of trees in the word and $\prec$ is given recursively by:
\begin{equation*}
s_1 \ldots s_k \prec t_1 \ldots t_n = \left((((s_1 \vee s_2 \vee \ldots \vee s_k \vee t_1) \vee t_2 ) \ldots )\vee t_n \right),
\end{equation*}
where $s_i$ and $t_i$ are trees and $t_1 \vee ... \vee t_n $ is the grafting of all the trees $t_i$ on a new root.

The coproduct is then given by considering the unique leftmost path from the root of the tree to the first node of arity different from two (this path can be trivial if the arity of the root is not two): there is as many term in the coproduct as there are edges in this path and a term is obtained from an edge $e$ by deleting all edges starting from vertices on the path between the root and $e$ and reordering the terms according to the previous recursive equation.

\begin{example}
An example of dipterous product and coproduct is presented below:
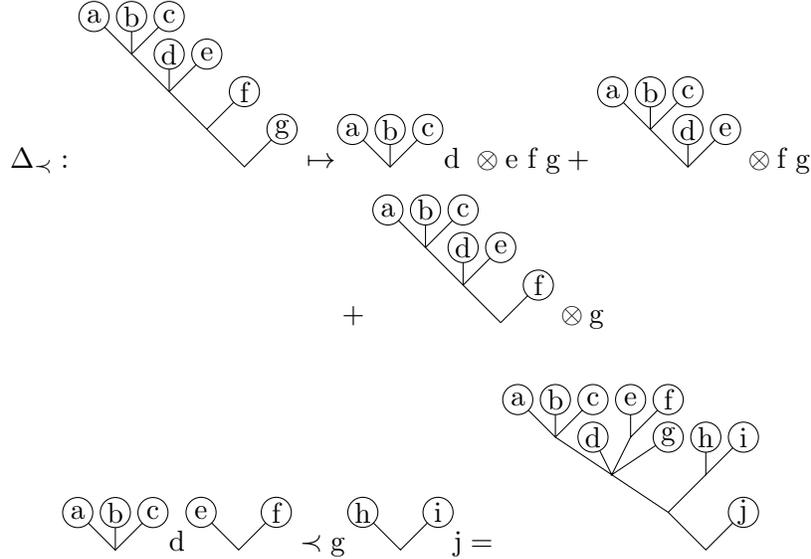
\begin{align*}
\Delta_{\prec}:
\begin{tikzpicture}[scale=0.5]
\draw (-2,2)--(0,0)--(1,1);
\draw (-1,1)--(0,2);
\draw (-1,3)--(-2,2)--(-3,3);
\draw (-2,2)--(-2,3);
\draw (-2,4)--(-3,3)--(-4,4);
\draw (-3,3)--(-3,4);
\draw[black, fill=white] (1,1) circle (0.4);
\draw[black, fill=white] (0,2) circle (0.4);
\draw[black, fill=white] (-2,3) circle (0.4);
\draw[black, fill=white] (-1,3) circle (0.4);
\draw[black, fill=white] (-2,4) circle (0.4);
\draw[black, fill=white] (-3,4) circle (0.4);
\draw[black, fill=white] (-4,4) circle (0.4);
\draw(1,1) node{g};
\draw(0,2) node{f};
\draw(-2,3) node{d};
\draw(-1,3) node{e};
\draw(-2,4) node{c};
\draw(-3,4) node{b};
\draw(-4,4) node{a};
\end{tikzpicture}
\mapsto &
\begin{tikzpicture}[scale=0.5]
\draw (-2,4)--(-3,3)--(-4,4);
\draw (-3,3)--(-3,4);
\draw[black, fill=white] (-2,4) circle (0.4);
\draw[black, fill=white] (-3,4) circle (0.4);
\draw[black, fill=white] (-4,4) circle (0.4);
\draw(-2,4) node{c};
\draw(-3,4) node{b};
\draw(-4,4) node{a};
\end{tikzpicture}
\text{d }
\otimes
\text{e f g}
 +
\begin{tikzpicture}[scale=0.5]
\draw (-1,3)--(-2,2)--(-3,3);
\draw (-2,2)--(-2,3);
\draw (-2,4)--(-3,3)--(-4,4);
\draw (-3,3)--(-3,4);
\draw[black, fill=white] (-2,3) circle (0.4);
\draw[black, fill=white] (-1,3) circle (0.4);
\draw[black, fill=white] (-2,4) circle (0.4);
\draw[black, fill=white] (-3,4) circle (0.4);
\draw[black, fill=white] (-4,4) circle (0.4);
\draw(-2,3) node{d};
\draw(-1,3) node{e};
\draw(-2,4) node{c};
\draw(-3,4) node{b};
\draw(-4,4) node{a};
\end{tikzpicture}
\otimes
\text{f g}\\
&
+
\begin{tikzpicture}[scale=0.5]
\draw (-2,2)--(-1,1);
\draw (-1,1)--(0,2);
\draw (-1,3)--(-2,2)--(-3,3);
\draw (-2,2)--(-2,3);
\draw (-2,4)--(-3,3)--(-4,4);
\draw (-3,3)--(-3,4);
\draw[black, fill=white] (0,2) circle (0.4);
\draw[black, fill=white] (-2,3) circle (0.4);
\draw[black, fill=white] (-1,3) circle (0.4);
\draw[black, fill=white] (-2,4) circle (0.4);
\draw[black, fill=white] (-3,4) circle (0.4);
\draw[black, fill=white] (-4,4) circle (0.4);
\draw(0,2) node{f};
\draw(-2,3) node{d};
\draw(-1,3) node{e};
\draw(-2,4) node{c};
\draw(-3,4) node{b};
\draw(-4,4) node{a};
\end{tikzpicture}
\otimes
\text{g}
\end{align*}

\begin{equation*}
\begin{tikzpicture}[scale=0.5]
\draw (-2,4)--(-3,3)--(-4,4);
\draw (-3,3)--(-3,4);
\draw[black, fill=white] (-2,4) circle (0.4);
\draw[black, fill=white] (-3,4) circle (0.4);
\draw[black, fill=white] (-4,4) circle (0.4);
\draw(-2,4) node{c};
\draw(-3,4) node{b};
\draw(-4,4) node{a};
\end{tikzpicture}
\text{d} 
\begin{tikzpicture}[scale=0.5]
\draw (-1,1)--(0,0)--(1,1);
\draw[black, fill=white] (1,1) circle (0.4);
\draw[black, fill=white] (-1,1) circle (0.4);
\draw(1,1) node{f};
\draw(-1,1) node{e};
\end{tikzpicture}
\prec
\text{g}
\begin{tikzpicture}[scale=0.5]
\draw (-1,1)--(0,0)--(1,1);
\draw[black, fill=white] (1,1) circle (0.4);
\draw[black, fill=white] (-1,1) circle (0.4);
\draw(1,1) node{i};
\draw(-1,1) node{h};
\end{tikzpicture}
\text{j}
=
\begin{tikzpicture}[scale=0.5]
\draw (1,1)--(0,0)--(-1,1);
\draw (-2.5,2)--(-1,1)--(0,2);
\draw (0,3)--(0,2)--(1,3);
\draw (-4,3)--(-2.5,2)--(-3,3);
\draw (-2,3)--(-2.5,2)--(-1,3);
\draw (-2,4)--(-2,3)--(-1,4);
\draw (-5,4)--(-4,3)--(-4,4);
\draw (-4,3)--(-3,4);
\draw[black, fill=white] (1,1) circle (0.4);
\draw[black, fill=white] (0,3) circle (0.4);
\draw[black, fill=white] (1,3) circle (0.4);
\draw[black, fill=white] (-3,3) circle (0.4);
\draw[black, fill=white] (-1,3) circle (0.4);
\draw[black, fill=white] (-5,4) circle (0.4);
\draw[black, fill=white] (-4,4) circle (0.4);
\draw[black, fill=white] (-3,4) circle (0.4);
\draw[black, fill=white] (-2,4) circle (0.4);
\draw[black, fill=white] (-1,4) circle (0.4);
\draw(1,1) node{j};
\draw(0,3) node{h};
\draw(1,3) node{i};
\draw(-3,3) node{d};
\draw(-1,3) node{g};
\draw(-5,4) node{a};
\draw(-4,4) node{b};
\draw(-3,4) node{c};
\draw(-2,4) node{e};
\draw(-1,4) node{f};
\end{tikzpicture}
\end{equation*}
\end{example} 

The products and their dual coproducts are then linked by the following mixed distributive law according to where the edge $e$ comes from:
\begin{align*}
\Delta_\star(u \star v)&= u \otimes v + \left(\Delta_\star(u)\right)_1 \otimes \left(\left(\Delta_\star(u)\right)_2 \star v \right) + \left(u \star \left(\Delta_\star(v)\right)_1 \right)\otimes \left(\Delta_\star(v)\right)_2\text{ (n.u.i.)}\\
\Delta_\star(u \prec v)&= 0\\
\Delta_\prec(u \star v)&= 0\\
\Delta_\prec(u \prec v)&= u \otimes v +  \left(\Delta_\prec(u)\right)_1 \otimes \left(  \left(\Delta_\prec(u)\right)_1 \star v \right) + \left( u \prec  \left(\Delta_\star(v)\right)_1 \right) \otimes  \left(\Delta_\star(v)\right)_2,
\end{align*}
where $\Delta_{\divideontimes}(u) = \left(\Delta_{\divideontimes}(u)\right)_1 \otimes \left(\Delta_{\divideontimes}(u)\right)_2$ for any operation ${\divideontimes}$ (Sweedler's notation of coproduct).

\bibliographystyle{alpha}
\bibliography{bibli}

\end{document}